\documentclass[11pt,reqno]{amsart}
\usepackage{longtable}

\usepackage{url}
\usepackage{enumerate}
\usepackage{amsmath, amsfonts, amsthm, bm}
\usepackage{amssymb}
\usepackage{cite}
\usepackage[pagewise]{lineno}
\usepackage{appendix}
\usepackage{xcolor}
\numberwithin{equation}{section}
\usepackage{amsmath}
\usepackage{graphicx}
\usepackage{graphicx}
\usepackage{hyperref}
\usepackage{color}
\usepackage{soul,xcolor} 
\usepackage{cite}
\usepackage{comment}
\usepackage{tikz-cd}
\usepackage{tikz}
\usetikzlibrary{positioning}

\usepackage{mathtools}

\usepackage[numeric]{amsrefs}

\allowdisplaybreaks

\theoremstyle{definition}
\newtheorem{thm}{Theorem}[section]
\newtheorem{cor}[thm]{Corollary}
\newtheorem{lem}[thm]{Lemma}
\newtheorem{exa}[thm]{Example}
\newtheorem{prop}[thm]{Proposition}
\newtheorem{defi}[thm]{Definition}
\newtheorem{rem}[thm]{Remark}

\newtheorem{thmA}{Theorem}

\newcommand{\indeg}{\operatorname{indeg}(\operatorname{Syz}(J_F))}

\title{On the Bourbaki degree of Plane Projective Curves}

\author{Roberto Alvarenga}
\address{\rm Roberto Alvarenga, São Paulo State University (UNESP), São José do Rio Preto, SP, Brazil}
\email{roberto.alvarenga@unesp.br}

\author{Murillo Lozano}
\address{\rm Murillo Lozano, São Paulo State University (UNESP), São José do Rio Preto, SP, Brazil}
\email{murillo.lozano@unesp.br}

\author{Parham Salehyan}
\address{\rm Parham Salehyan, São Paulo State University (UNESP), São José do Rio Preto, SP, Brazil}
\email{p.salehyan@unesp.br}

\begin{document}

\begin{abstract}
Let $F$ be a reduced plane projective curve defined over an algebraically closed field $k$.  The Bourbaki degree of $F$, denoted by $\mathrm{Bour}(F)$, measures its freeness and was introduced in \cite{JNS2024}. It is defined as the degree of $R/I_\epsilon$, where $R = k[x,y,z]$ and $I_\epsilon \subseteq R$ is the Bourbaki ideal associated with a minimal generator $\epsilon$ of the module of first syzygies of the Jacobian ideal $J_F$. 
In this article, we investigate the local Bourbaki degree and apply it to refine the well-known upper bound for the global Bourbaki degree. We then study curves with prescribed Bourbaki degree through the minimal graded free resolution of $R/J_F$, describing the possible resolution forms and expressing $\mathrm{Bour}(F)$ in terms of the corresponding graded shifts. 
Finally, we completely solve the realization problem in fixed degree: for every $d\geq 4$, we prove that all the admissible Bourbaki degrees for a $d+1$ degree curve are indeed realized.

\end{abstract}
	\maketitle


	\section{Introduction}

Let $R$ be a commutative Noetherian ring and let $M$ be a finitely generated $R$-module. A \textit{Bourbaki sequence} of $M$ is a short exact sequence
\[
0\longrightarrow N\longrightarrow M\longrightarrow I\longrightarrow0,
\]
where $N$ is a free $R$-module and $I$ is an ideal of $R$, called a \textit{Bourbaki ideal} of $M$. Bourbaki sequences and Bourbaki ideals have been studied in several contexts, see for instance \cites{Auslander,Evans,Herzog2,Miller, Dimca4}.

In the context of plane projective curves, Jardim, N. Nejad and Simis introduced in \cite{JNS2024} a numerical invariant arising from a Bourbaki sequence associated with the module of syzygies of the Jacobian ideal. More precisely, let $F\in R:=k[x,y,z]$ be a reduced homogeneous polynomial of degree $d+1$, where $k$ is algebraically closed of $\operatorname{char}(k)=0 \text{ or } \operatorname{char}(k)\gg 0$   . Let $J_F=\langle F_x,F_y,F_z\rangle$ be its Jacobian ideal. 
Let $e:=\operatorname{indeg}(\operatorname{Syz}(J_F))$ be the minimal degree of a nonzero syzygy $\epsilon\in\operatorname{Syz}(J_F)$. Then there exists a homogeneous ideal $I_\epsilon\subset R$ such that
\[
0\longrightarrow R(-e)\longrightarrow\operatorname{Syz}(J_F)\longrightarrow I_\epsilon(e-d)\longrightarrow0
\]
is exact. The ideal $I_\epsilon$ is the Bourbaki ideal associated with $\epsilon$, and the \textit{Bourbaki degree} of $F$ is defined by $\mathrm{Bour}(F)=\deg(R/I_\epsilon)$. By \cite{JNS2024}*{Thm. 2.1},
\[
\mathrm{Bour}(F)=d^2+e(e-d)-\tau(F),
\]
where $\tau(F)$ denotes the global Tjurina number of $F$. Note that we abuse the notation and denote by $F$ both the polynomial and the curve it defines.

In \cite{JardimMonteiroNejadRamos2026} the Bourbaki degree is extended to arbitrary three-generated homogeneous ideals and several related questions are addressed. Our focus remains on Jacobian ideals of plane projective curves, particularly on local properties of the Bourbaki degree, the structure of the minimal free resolution when $\mathrm{Bour}(F)$ is prescribed, refined upper bounds for it, and the realization problem in fixed degree. In the following we describe more explicitly our main contributions.

We first consider the local behavior of the Bourbaki degree. By \cite{JNS2024}*{Thm. 2.1}, $V(I_\epsilon)$ is a finite set of points in $\mathbb{P}^2$. Let $P \in V(I_\epsilon)$ and let $\mathrm{Bour}_{P} (F, \epsilon)$ be the dimension of $(R/I_\epsilon)_P$. We observe in Proposition \ref{cor3} that
\[ 
\mathrm{Bour}(F)=\sum_{P\in\mathbb P^2}\mathrm{Bour}_{P} (F, \epsilon)
\]
and call it the \textit{local Bourbaki formula}. 
We then apply it to show, among other things, that for a fixed $n \in \mathbb{N}$, there exists an irreducible curve $F$ such that $\mathrm{Bour}(F) = n$. This is Theorem \ref{prop-exis-grau-varia}.

The local Bourbaki formula is useful in the following sense. In order to compute the Bourbaki degree using the definition, whether via the formula involving the Tjurina number or directly via the local formula, one must determine the generators of $\mathrm{Syz}(J_F)$ in all cases.
The key difference, however, lies in the fact that, in general, since the degrees of the generators of $I_{\epsilon}$ are smaller than $\deg(F)$, one has $|V(I_\epsilon)| \leq |V(J_F)|$ for singular curves. In other words, the zero set of the Bourbaki ideal is typically smaller than the set of singular points of the curve. Consequently, fewer computations of vector space dimensions are required when applying the local formula. This reduction in the number of points to be analyzed makes the local approach computationally advantageous, especially for curves with many singularities.

We then study curves with prescribed Bourbaki degree through their minimal graded free resolutions. Since $I_\epsilon$ is a perfect ideal of height two when $F$ is not free, its resolution is governed by the Hilbert--Burch Theorem. For each positive integer $n$, in Theorem \ref{prop:classification-general}, we give a numerical characterization of the possible resolution types associated with curves satisfying $\mathrm{Bour}(F)=n$, as follows: 

\begin{thmA}
Let $F$ be singular and not free.  Then $\operatorname{Bour}(F)=:n$ if and only if there exist an integer $m\geq 2$, positive integers $c_1,\ldots,c_{m-1}$, and nonnegative integers $u_1,\ldots,u_{m-1}$ satisfying
	\[ 2\leq m\leq \left\lfloor\frac{1+\sqrt{1+8n}}{2}\right\rfloor,\qquad 0\leq u_1\leq\cdots\leq u_{m-1},\qquad u_1+c_1\leq\cdots\leq u_{m-1}+c_{m-1},
	\]
	and
	\[ n=\frac{1}{2}\left[\left(\sum_{j=1}^{m-1} c_j\right) ^2+\sum_{j=1}^{m-1}c_j^2\right]
	+\sum_{j=1}^{m-1}u_jc_j,
	\]
such that, setting $s :=\sum_{j=1}^{m-1}c_j$, the ring $R/J_F$ has a minimal free resolution of the form
\[
	\begin{aligned}
		0\longrightarrow&
		\bigoplus_{j=1}^{m-1}R\bigl(-(2d-e+s+u_j+c_j)\bigr)\\
		\longrightarrow&
		R(-(d+e))\oplus R(-(2d-e+s))
		\oplus\bigoplus_{j=1}^{m-1}R\bigl(-(2d-e+s+u_j)\bigr)\\
		\longrightarrow& R^3(-d)\longrightarrow R\longrightarrow R/J_F\longrightarrow 0.
	\end{aligned}
	\]
In this case, necessarily $e \leq (d+s)/2$.
\end{thmA}

We next apply the local Bourbaki degree formula to obtain refined  upper bounds for $\mathrm{Bour}(F)$, improving \cite{JNS2024}*{Thm. 2.10}. Writing $\epsilon=(a,b,c)$, let $H_{\epsilon} :=\langle a,b,c\rangle \subseteq R$. We compare $H_\epsilon$ with the Bourbaki ideal $I_\epsilon$. When $\{ a,b,c\}$ is linearly independent, the structure of $H_\epsilon$ yields stronger estimates: if  $H_\epsilon$ is perfect, then  Theorem  \ref{prop:perfect-bound} yields
\[ \mathrm{Bour}(F)\leq e^2-e+1,\] 
while in the non-perfect case, we show in Proposition \ref{prop:nonperfect-bound} that
\[\mathrm{Bour}(F)\leq e^2-e.\] 
The former bound is also obtained in \cite{JardimMonteiroNejadRamos2026} for a three-equigenerated ideal. We further refine these bounds  as follows: 

\begin{thmA} Let $\ell_{\epsilon} := | V(H_\epsilon) \setminus V(I_\epsilon)|.$ Then
\[ \operatorname{Bour}(F)\le
\begin{cases}
 e^2-e+1-\ell_{\epsilon},& H_\epsilon\text{ perfect},\\[2mm]
 e^2-e-\ell_{\epsilon},& H_\epsilon\text{ non-perfect}.
\end{cases}
\]
Consequently, in both cases,
\[ \operatorname{Bour}(F) \le e^2-e+1-\ell_{\epsilon}.
\]
\end{thmA} 
The previous result is Corollary \ref{cor:refined-bound}.

Finally, we consider the realization problem in fixed degree. For a reduced plane curve $F$ of degree $d+1$, one has $0\leq\operatorname{Bour}(F)\leq d^2$, see \cite{JNS2024}*{Thm. 2.10}. For $d=0,1$, all the admissible Bourbaki degrees are realized. For cubic curves ($d=2$) and quartic curves ($d=3$), the du Plessis--Wall bounds exclude the Bourbaki degrees $2$ and $5$, respectively, and all other admissible Bourbaki degree are realized. We discuss these cases in more details in section \ref{sec:realization}. The following theorem gives a complete answer in all higher degrees:

\begin{thmA} \label{thm-C}
    Let $d\geq 4$ be an integer. Then, for every integer $n$ with $0\leq n\leq d^2$, there exists a reduced projective plane curve $F$ of degree $d+1$ such that $\operatorname{Bour}(F)=n$.
\end{thmA}

This is Theorem \ref{teo:realizacao-completa-bourbaki} and, in other words, shows that all admissible Bourbaki degrees are realized for fixed $d \geq 4$. Finally, we remark that the degree assumption is sharp, as the cubic and quartic cases show.


\section{Preliminaries}

Throughout the article, $k$ denotes an algebraically closed field, and $R := k[x,y,z]$ is the polynomial ring in the variables $x,y,z$ with coefficients in $k$. Let $F \in R$ be a reduced homogeneous polynomial of degree $d+1$. We assume that  $\operatorname{char}(k)=0 \text{ or } \operatorname{char}(k)\gg 0$ . By abuse of notation, we denote by $F$ both the polynomial and the curve it defines in the projective plane $\mathbb{P}^2$ over $k$.

The ring $R$ and the $R$-modules we present here are all graded by degree in the most natural sense, i.e., the grading is the standard one given by the total degree in $x,y,z$. We denote by $J_F$ the Jacobian ideal of $F$, which is generated by the partial derivatives $F_x, F_y, F_z$.

\subsection{Construction of the Bourbaki ideal \texorpdfstring{$I_\epsilon$}{I-epsilon} }

We consider the finitely generated graded $R$-module of first syzygies of $J_F$,
\[\mathrm{Syz}(J_F) = \{ (a,b,c) \in R^3 \mid aF_x + bF_y + cF_z=0 \}.\]
Let $\{ \epsilon_1, \ldots, \epsilon_k \}$ be a set of homogeneous generators of $\mathrm{Syz}(J_F)$, where $e_i = \deg \epsilon_i$, $i=1,\ldots,k$. Without loss of generality, we assume the degrees are ordered increasingly: $e_1 \leq e_2 \leq \cdots \leq e_k$.
Hence, $\epsilon_1$ is a generating syzygy of minimal degree, and we call it a \textit{minimal syzygy}. For simplicity, we set $\epsilon := \epsilon_1$ and denote its degree by $e = e_1$. 
 We have the following exact sequence of graded $R$-modules
\[	0 \longrightarrow R(-e) \stackrel{\epsilon}{\longrightarrow} \mathrm{Syz}(J_F) \longrightarrow \mathrm{coker}(\epsilon) \longrightarrow 0. \tag{1}\]
Setting $M_{\epsilon} := \mathrm{coker}(\epsilon)$, one has $M_{\epsilon} = \langle \delta_2, \ldots, \delta_k\rangle$, where $\delta_i \equiv \epsilon_{i} (\bmod\hspace{0.1cm} \epsilon)$, for $i=2,\ldots,k$. Thus $M_{\epsilon}$ is generated by the images of the remaining syzygies modulo the submodule generated by $\epsilon$.

For the sake of completeness, in the following we prove that $M_{\epsilon}$ is isomorphic to an ideal of $R$, which we shall denote by $I_{\epsilon}$. Moreover, we present an explicit description of $I_\epsilon$. We emphasize the construction and description of $I_\epsilon$, since it plays a key role in defining  the local Bourbaki degree and obtaining improved bounds for the Bourbaki degree. Consider the following commutative diagram
    	\begin{center}
    	  \begin{tikzcd}[row sep=normal, column sep=normal]
	&	0 \arrow[d] & 0 \arrow[d] \\
		&R(-e) \arrow[d, "\epsilon"'] \arrow[r, equal] & R(-e) \arrow[d, "\tilde{\epsilon}"] \\
		0 \arrow[r] & \mathrm{Syz}(J_F) \arrow[r] \arrow[d] & R^3 \arrow[r, "\theta_F"] \arrow[d] & J_F(d) \arrow[r] \arrow[d, equal] & 0 \\
		0 \arrow[r] & M_\epsilon \arrow[r] \arrow[d] & R^3 / \tilde{\epsilon}(R(-e)) \arrow[r, "\tilde{\theta}_F"] \arrow[d] & J_F(d) \arrow[r] & 0 \\
		& 0 & 0 
	\end{tikzcd},
    	\end{center}
where
\begin{center}
    $\tilde{\epsilon} : R(-e) \overset{\epsilon}{\longrightarrow} \mathrm{Syz}(J_F) \xhookrightarrow{\hspace{0.5cm}} R^3$\quad \text{ and } \quad 
$\theta_F (g_1, g_2, g_3) = g_1F_x + g_2F_y + g_3F_z $,
\end{center}
while $\tilde{\theta}_F$ has the same rule as $\theta_F.$ Hence,
\begin{align*}
    \tilde{\epsilon}(R(-e)) = \mathrm{Im}(\tilde{\epsilon}) = \{ \tilde{\epsilon}(g) \mid  g \in R(-e) \}
    = \{ \epsilon(g) \mid g \in R(-e)\}
    = \{ \epsilon\cdot g \mid g \in R(-e) \}
    = \langle \epsilon \rangle.
\end{align*}
Therefore, the lower horizontal exact sequence of the diagram can be written as
\begin{align*}
     0 \longrightarrow M_\epsilon \longrightarrow R^3 / \langle \epsilon \rangle \longrightarrow J_F (d) \longrightarrow 0.
\end{align*}
We prove that $R^3 / \langle \epsilon \rangle$ is a torsion-free $R$-module. First, note that since $\epsilon$ is a minimal generator of $\mathrm{Syz}(J_F)$, writing $\epsilon = (\epsilon_1, \epsilon_2, \epsilon_3)$, its coordinates have no common factor of degree $\geq 1$. Indeed, if there were a factor $h \in R$, with $\deg h \geq 1$, dividing $\epsilon_i$ for all $i = 1,2,3$, then we could write $\epsilon  = (\epsilon_1, \epsilon_2, \epsilon_3) = (h\cdot \epsilon_1^{\prime}, h\cdot \epsilon_2^{\prime}, h\cdot \epsilon_3^{\prime}) = h\cdot (\epsilon_1^{\prime},\epsilon_2^{\prime},\epsilon_3^{\prime}). $ Thus, since $\epsilon$ is a syzygy of $J_F$, $(\epsilon_1^{\prime},\epsilon_2^{\prime},\epsilon_3^{\prime})$ would also be a syzygy of $J_F$, of degree less than $e$, contradicting the minimality of $\epsilon$.

Let $r\in R\setminus \{0 \}$  and $u = (u_1, u_2, u_3) \in R^3$ such that $r\cdot u = 0$ in $R^3 / \langle \epsilon \rangle$, i.e., $r\cdot u \in \langle \epsilon \rangle$. Then $r \cdot u = s \cdot \epsilon$, with $s \in R(-e)$. Let $\tilde{d} = \gcd(r,s)$ and write $r = \tilde{d} \cdot \tilde{r}$ and $s = \tilde{d} \cdot \tilde{s}$. Since $\gcd(\tilde{r}, \tilde{s}) = 1$ and $(\tilde{d} \tilde{r} u_1,\tilde{d} \tilde{r} u_2, \tilde{d} \tilde{r} u_3 ) = (\tilde{d} \tilde{s} \epsilon_1, \tilde{d} \tilde{s} \epsilon_2, \tilde{d} \tilde{s} \epsilon_3)$, it follows that $\tilde{r} \mid \epsilon_i, \text{ for } i =1,2,3$. Hence, $\tilde{r}$ must be invertible, thus $u_i = (\tilde{s}/\tilde{r})\cdot \epsilon_i$, for all $i=1,2,3$, and hence $u = (\tilde{s}/\tilde{r})\cdot \epsilon \in \langle \epsilon \rangle$. Therefore, $R^3 / \langle \epsilon \rangle$ is torsion-free. 

Since $M_\epsilon$ is a submodule of $R^3 / \langle \epsilon \rangle$, it is also torsion-free. By \cite{Aron}*{Prop. 3.5.3}, it follows that $M_\epsilon$ is isomorphic to an ideal of $R$. In this proof, the necessity of the minimality of $\epsilon$ becomes clear.

Since we have the graded isomorphism $M_{\epsilon} \simeq I_{\epsilon}(e - d)$ (see for instance \cite{JNS2024}*{Thm. 2.1, (b)}), the exact sequence (1) can be rewritten as
\[ 0 \longrightarrow R(-e) \stackrel{\epsilon}{\longrightarrow} \mathrm{Syz}(J_F) \longrightarrow I_{\epsilon} (e-d)\longrightarrow 0. \tag{2}\]
Thus, we obtain a Bourbaki sequence for the module $\mathrm{Syz}(J_F)$, and hence $I_{\epsilon}$ is indeed a Bourbaki ideal. The Bourbaki degree of $F$ is defined as the degree of the $R$-module $R/I_\epsilon$.

\begin{rem}
    In the particular case where $F$ is a free curve (i.e., $\mathrm{Syz}(J_F)$ is a free $R$-module), \cite{JNS2024}*{Thm. 2.1(a)} yields $\deg(R/J_F) = d^2 + e(e-d)$, and therefore $\mathrm{Bour}(F)=0$.  Free curves are extensively discussed in the literature, since their syzygy module is the simplest possible. See, for instance, \cite{Dimca1} and \cite{Dimca2}.
\end{rem}   

\subsection{Effective computation of the generators of the Bourbaki ideal \texorpdfstring{$I_\epsilon$}{I-epsilon}}	

We now give an effective procedure for determining the generators of the Bourbaki ideal directly from the relations among the minimal generators of
$\operatorname{Syz}(J_F)$. In what follows we suppose $F$ is not free. Since 
$\{\epsilon_1,\ldots,\epsilon_k\}$ is a minimal set of
generators of $\operatorname{Syz}(J_F)$, we have $k\geq3$. Hence $\operatorname{Syz}(J_F)$ has a minimal free resolution given by
\[ 0\longrightarrow F_1\longrightarrow
\bigoplus_{j=1}^{k}R(-e_j)
\longrightarrow \operatorname{Syz}(J_F)\longrightarrow 0. \]
Since $\operatorname{rank}(\operatorname{Syz}(J_F))=2$, by the additivity of rank, $\operatorname{rank}(F_1)=k-2$. Therefore, there are exactly $k-2$ minimal relations among the generators $\epsilon_1,\ldots,\epsilon_k$, which can be written in the form
\[ \sum_{j=1}^{k}c_{ij}\epsilon_j=0,
\qquad i=1,\ldots,k-2,\]
where the $c_{ij}\in R$ are homogeneous polynomials. In $M_\epsilon$, 
\[ \sum_{j=2}^{k}c_{ij}\delta_j=0, \qquad i=1,\ldots,k-2.
\]
For each $i=1,\ldots,k-2$, let $s_i$ denote the degree of the relation $\sum_{j=2}^{k}c_{ij}\delta_j=0$. 
By \cite{JNS2024}*{Thm. 2.1(c)}, we obtain a minimal free
resolution of $I_\epsilon$ given by
\[ 0\longrightarrow \bigoplus_{i=1}^{k-2}R(-s_i+d-e)
\xrightarrow{\ A^t\ } \bigoplus_{j=2}^{k}R(-e_j+d-e)
\longrightarrow I_\epsilon\longrightarrow0,
\]
where
\[A=
\begin{bmatrix}
 c_{12} & c_{13} & \cdots & c_{1k}\\
 c_{22} & c_{23} & \cdots & c_{2k}\\
 \vdots & \vdots & \ddots & \vdots\\
 c_{k-2,2} & c_{k-2,3} & \cdots & c_{k-2,k}
\end{bmatrix}_{(k-2)\times(k-1)}.
\]
By the Hilbert-Burch Theorem, $I_\epsilon$ is generated by the maximal minors of $A$. More precisely, for $j=2,\ldots,k$, let $A^{[j-1]}$ be the submatrix of $A$ obtained by deleting the $(j-1)$-th column. Defining
\[ \omega_j=(-1)^j\det\bigl(A^{[j-1]}\bigr),
\qquad j=2,\ldots,k,\]
we obtain the explicit description
\[ I_\epsilon=\langle\omega_2,\ldots,\omega_k\rangle.\]
Thus, once the minimal relations among the generators of
$\operatorname{Syz}(J_F)$ are known, the generators of the Bourbaki ideal can be computed directly as maximal minors of the matrix obtained by deleting the column corresponding to the chosen minimal syzygy $\epsilon$.

\subsection{Free resolutions associated with the Bourbaki ideal} \label{subsec:prelim-resolutions}

We record from \cite{JNS2024} the correspondence between the graded shifts in a minimal free resolution of the Bourbaki ideal and those in a minimal free resolution of $\operatorname{Syz}(J_F)$. Suppose that $F$ is not free and that the
Bourbaki ideal associated with a minimal syzygy $\epsilon$ of degree $e$ has minimal graded free resolution
\[ 0\longrightarrow \bigoplus_{j=1}^{m-1}R(-a_j)
\longrightarrow \bigoplus_{i=1}^{m}R(-b_i)
\longrightarrow I_\epsilon\longrightarrow0. \]
Then \cite{JNS2024}*{Thm. 2.1(c)} gives a minimal graded free resolution
\[ 0\longrightarrow \bigoplus_{j=1}^{m-1}R\bigl(-(d-e+a_j)\bigr) \longrightarrow R(-e)\oplus \bigoplus_{i=1}^{m}R\bigl(-(d-e+b_i)\bigr)
\longrightarrow\operatorname{Syz}(J_F)\longrightarrow0.
\]
Consequently, after shifting by $-d$ and adjoining the natural map $R^3(-d)\to R$, one obtains the corresponding minimal free resolution of $R/J_F$. This correspondence will be used repeatedly in Sections 4 and 6.


\section{The Local Bourbaki Degree}
	
In order to introduce the Bourbaki degree at a point $P \in \mathbb{P}^2$ we begin with a remark about homogeneous ideals and their zeros in projective space.

\begin{rem}\label{obs1}
		
Let $I$ be an ideal of $R$ with  $V(I)=\{P_1,\dots,P_r\}$ the set of zeros of $I$ in $\mathbb{P}^2$, where $P_i=(a_i:b_i:c_i)$ for $i=1,\dots,r$. We observe that after a  projective change of coordinates, each $P_i$ can be chosen with its third coordinate equal to $1$. Indeed, consider $s: ax+by+cz=0$ a line in $\mathbb{P}^2$ such that $P_i \notin s$, for all $i=1,\ldots,r$. Thus in the new coordinates given by
\begin{align*}
	\begin{cases}
		x'=x\\
		y'=y\\
		z'=ax+by+cz
	\end{cases},
\end{align*}
each point $P_i$ has $z\neq0$, as desired. Therefore, without loss of generality, we may assume $P_i = (a_i:b_i:1)$, for all $i=1,\ldots,r$. In the following, we make this assumption without mentioning it. 
\end{rem}

The definition of the local Bourbaki degree is closely tied to the following result, which is valid for any homogeneous ideal $I$ of $k[x,y,z]$ such that $\operatorname{dim} R/I = 1$. After the proof, we will restrict ourselves to the case where the ideal $I$ in the following lemma is the Bourbaki ideal $I_\epsilon$ we constructed in the last section.

\begin{lem}\label{teo1}
Let $I \subseteq R$ be a homogeneous ideal such that $V(I)= \{ P_1, \ldots, P_r\} \subseteq \mathbb{P}^2$, where $P_i = (a_i:b_i:1).$  Then
\begin{align*}
	\deg (R/I) = \displaystyle\sum_{i=1}^{r} \dim\left( \dfrac{k[x,y]}{J}\right)_{\langle x-a_i,y-b_i\rangle},
\end{align*}
where $J$ is the dehomogenization of $I$ with respect to $z$.
\end{lem}

\begin{proof}
Define the ring morphism
		\begin{align*}
			\varphi:  \hspace{0.1cm} &k[x,y,z] \longrightarrow k[x,y]\\
			&f(x,y,z) \longmapsto f(x,y,1).
		\end{align*}
Let $J := \varphi(I)$, i.e., the dehomogenization of $I$ with respect to $z$. Since $\varphi$ is surjective, $J$ is an ideal of $k[x,y].$ Therefore,  $V(J) = \{ (a_1,b_1), \ldots, (a_r,b_r)\}$.

 Let $(R/I)_n$ denote the homogeneous component of degree $n$ of the graded ring $R/I$. Since $\dim R/I = 1$, for $n \gg 0$, 
 \[ \mathrm{HP}_{R/I}(n) = \ell((R/I)_n)=\deg R/I,\] where $\mathrm{HP}_{R/I}$ denotes the Hilbert polynomial of $R/I$ and $\ell((R/I)_n)$ is the length of $(R/I)_n$.  We define
\begin{align*}
	\psi_n: \hspace{0.1cm}&(R/I)_n \longrightarrow k[x,y]/J \\
	&f(x,y,z)+I \longmapsto f(x,y,1)+J.
\end{align*}

It is straightforward to verify that $\psi_n$ is well-defined and is $k$-linear. Furthermore, for $n \gg 0$,  $\psi_n$ is an isomorphism of $k$-vector spaces.

Thus $\dim_k (R/I)_n = \dim_k (k[x,y]/J)$. Consequently, for $n\gg 0$,
\begin{align*}
\deg(R/I) = \ell((R/I)_n) = \dim_k (R/I)_n = \dim_k (k[x,y]/J).
\end{align*}
		  
Let $m_i'= \langle x-a_i,y-b_i\rangle$ be the maximal ideal of $k[x,y]$ associated to the points $P_i '=(a_i,b_i)$, $i=1,\ldots,r$. Since $A=k[x,y]/J$ is Artinian, by \cite{Atiyah}*{Thm. 8.7},
\[A \simeq A_{m_1'} \times \cdots \times A_{m_r'}, \]
where each $A_{m_i '}$ is a local Artinian ring. Hence $\dim_k A = \sum_{i=1}^{r} \dim_k (A_{m_i '})$. See also \cite{Fulton}*{Chap. 2 Prop. 6} for a proof of the preceding isomorphism and identity. Therefore,
\begin{align*}
\deg(R/I) = \displaystyle\sum_{i=1}^{r} \dim_k A_{m_i '} = \displaystyle\sum_{i=1}^{r} \dim_k \left( \dfrac{k[x,y]}{J}\right)_{m_i'}=\displaystyle\sum_{i=1}^{r} \dim_k\left( \dfrac{k[x,y]}{J}\right)_{\langle x-a_i,y-b_i\rangle},
\end{align*}
which completes the proof.\end{proof}

Above lemma can be seen as a version of the associativity formula, see \cite{Cohen-Macauly-Rings}*{Cor. 4.7.8}. 
	
Now we consider the local Bourbaki degree at a point $P$ of a non-free projective plane curve $F$.  Since $I_\epsilon$ has codimension $2$ (see for instance \cite{JNS2024}*{Thm. 2.1}), the ideal $I_\epsilon$ has finitely many zeros. Let $V(I_\epsilon) = \{ P_1, \ldots, P_s\}\subseteq \mathbb{P}^2$. Assuming $P_i = (a_i:b_i:1)$,  Lemma \ref{teo1} yields that 
	\[ \mathrm{Bour}(F) = \deg (R/I_\epsilon) = \displaystyle\sum_{i=1}^{s} \dim_k \left( \dfrac{k[x,y]}{I_\epsilon(x,y,1)}\right)_{\langle x-a_i,y-b_i\rangle},\tag{3}\]
where $I_\epsilon (x,y,1) $ is the ideal $I_\epsilon$ dehomogenized with respect to $z$. This motivates the following definition.
	
\begin{defi}\label{def1}
In the above notation, we define the local Bourbaki degree of $F$ at $P = (a:b:1) \in V(I_\epsilon)$ by 
		\begin{align*}
			\mathrm{Bour}_{P} (F, \epsilon) = \dim_k \left(\dfrac{k[x,y]}{I_\epsilon(x,y,1)}\right)_{\langle x-a,y-b\rangle}. 
		\end{align*}
\end{defi}	

The previous lemma yields:

 \begin{prop}(Local Bourbaki Formula).\label{cor3} 
     $\mathrm{Bour}(F) = \displaystyle\sum_{P \in V(I_\epsilon)} \mathrm{Bour}_{P} (F, \epsilon)$.
 \end{prop}

\begin{proof}
 The result follows from (3) and Definition \ref{def1}.
 \end{proof}

\begin{rem}\label{obs2}
Let $P = (a:b:1) \notin V(I_\epsilon)$. Then, 
\[ \dim_{k}\left(\dfrac{k[x,y]_{\langle x-a,y-b\rangle}}{I_\epsilon(x,y,1)_{\langle x-a,y-b\rangle}}\right) =0. \]
Thus, we define $\mathrm{Bour}_{P} (F, \epsilon) = 0$. Hence we  can extend the definition of the local Bourbaki degree to all points in $\mathbb{P}^2$ and generalize Proposition \ref{cor3} as follows:
 \[ \mathrm{Bour}(F) = \displaystyle\sum_{P \in \mathbb{P}^2} \mathrm{Bour}_{P} (F, \epsilon).\]
Hence, we might consider the $2$-cycle on $\mathbb{P}^2$
\[  \displaystyle\sum_{P \in \mathbb{P}^2} \mathrm{Bour}_{P} (F, \epsilon)\cdot P\]
which is an effective divisor on each curve defined by the generators of $I_\epsilon$. 
\end{rem}

\begin{cor}\label{cor1}
	Let $P \in \mathbb{P}^2$. Then $P \in V(I_\epsilon) $ if and only if $\mathrm{Bour}_{P} (F, \epsilon) >0$. 
\end{cor}

The next corollary provides a lower bound for the Bourbaki degree, in terms of the number of elements of $V(I_\epsilon)$.

\begin{cor}\label{cor2}
	Let $\ell = |V(I_\epsilon)|$. Then $\mathrm{Bour}(F) \geq \ell$, with equality valid if and only if $\mathrm{Bour}_{P} (F, \epsilon) = 1$, for every $P \in V(I_\epsilon)$. In particular, $F$ is nearly-free if and only if $V(I_\epsilon) = \{ P_0\}$ and $\mathrm{Bour}_{P_0}(F, \epsilon) = 1$.
\end{cor}	

\begin{proof}
Let $V(I_\epsilon) = \{ P_1, \ldots, P_\ell\}$. By Corollary \ref{cor1}, $\mathrm{Bour}_{P_i}(F,\epsilon) \geq 1$ for all $i = 1,\ldots,\ell$. Thus \[\displaystyle\sum_{i=1}^{\ell} \mathrm{Bour}_{P_i}(F, \epsilon) \geq \ell,\] i.e., $\mathrm{Bour}(F) \geq \ell$. The remaining statements are obtained in a straightforward way.\end{proof}

    We now provide an example using the local formula we introduced.
    \begin{exa}\label{ex1}
		Computation of the Bourbaki degree of the nodal cubic $F: y^2z - x^3 - x^2z = 0$ using the local formula. By \cite{Singular}, $\mathrm{Syz}(J_F) = \langle \epsilon_1, \epsilon_2, \epsilon_3, \epsilon_4 \rangle$, with $\deg \epsilon_i = 2$, for $i=1,2,3,4$. 
		Moreover, these generators satisfy
		\begin{center}
			$\begin{cases}
				(3x+2z)\cdot \epsilon_1 + (-y)\cdot \epsilon_2 +(-2z) \cdot \epsilon_3=0\\[2mm]
				x\cdot\epsilon_2 + (-3y)\cdot \epsilon_3+z\cdot \epsilon_4=0
			\end{cases}$.
		\end{center}

Consider $M_{\epsilon_1} = \frac{\mathrm{Syz}(J_F)}{R(-2)\cdot \epsilon_1} = \langle \delta_2, \delta_3, \delta_4\rangle$, where $\delta_i \equiv \epsilon_{i} \pmod{\epsilon_1}, i=2,3,4$.
Thus, the matrix of the relations among $\delta_2,\delta_3,\delta_4$ is
\[ A=
\begin{bmatrix}
-y & -2z & 0\\
x & -3y & z
\end{bmatrix}.
\]
Hence, the generators of the Bourbaki ideal $I_{\epsilon_1}$ are the maximal minors of $A$. More precisely,
\[ \omega_2 = \det
\begin{bmatrix}
-2z & 0\\
-3y & z
\end{bmatrix}
=-2z^2, \hspace{0.2cm}
\omega_3
=
-\det
\begin{bmatrix}
-y & 0\\
x & z
\end{bmatrix}
=yz, \hspace{0.2cm}
\omega_4
=
\det
\begin{bmatrix}
-y & -2z\\
x & -3y
\end{bmatrix}
=2xz+3y^2.
\]
Therefore, $I_{\epsilon_1}=\langle 2z^2,yz,2xz+3y^2 \rangle.$
Computing $V(I_{\epsilon_1})$ gives only the point $(1:0:0)$. Since $I_{\epsilon_1}(1,y,z) = \langle z^2,\; yz,\; 2z+3y^2 \rangle$, we obtain 
		\begin{align*}
			\mathrm{Bour}(F) = \dim_{k}\left( \dfrac{k[y,z]}{I_{\epsilon_1}(1,y,z)}\right)_{\langle y,z\rangle}=3.
		\end{align*}
	\end{exa}

The next example is the computation of the Bourbaki degree of a curve using all available methods. We see that, indeed, the number of singular points is greater than the number of points in $V(I_\epsilon)$.
	
\begin{exa}\label{ex2}
We compute the Bourbaki degree of the curve $F: y^2z^2 - x^4 + 2x^3z - x^2z^2 = 0$, using each of the three available methods.
The Jacobian ideal of $F$ is given by $J_F = \langle yz^2, x^2z+2y^2z-xz^2, x^3+3y^2z-xz^2 \rangle$. Hence, the singular points of $F$ are  $P_1 = (0:0:1)$, $P_2 = (0:1:0)$, and $P_3 = (1:0:1)$.
By \cite{Singular}, $\tau_{P_1} (F) = 1$, $\tau_{P_2}(F) = 3$, and $\tau_{P_3}(F)=1$. Therefore, $\tau(F) = 5$.
Using the software again, we obtain $\mathrm{Syz}(J_F) = \langle\epsilon_1, \epsilon_2, \epsilon_3 \rangle$, with $2 =\deg \epsilon_1 \leq \deg \epsilon_2 \leq \deg \epsilon_3=3$ and
\begin{center}
$(x^2+2y^2-xz)\cdot \epsilon_1 + (-2xy+yz)\cdot\epsilon_2 + (z)\cdot \epsilon_3=0. $
\end{center}
Thus, we have $d = 3$ and $e=2$. Hence, by the Bourbaki degree formula,
\begin{center}
$ \mathrm{Bour}(F) = 3^2 + 2^2 - 2\cdot3-5 = 2.$
\end{center}
Alternatively, we determine the Bourbaki ideal associated to the minimal syzygy $\epsilon_1$. Considering $\epsilon_1$ as a minimal generator, one has 
$ M_{\epsilon_1} = \frac{\mathrm{Syz}(J_F)}{R(-2)\cdot \epsilon_1} = \langle \delta_2, \delta_3 \rangle$,
where $\delta_i \equiv \epsilon_{i} (\bmod \hspace{0.1cm} \epsilon_1) $. Since in this case, the matrix of the relations among $\delta_2, \delta_3$ is given by
\[
B=
\begin{bmatrix}
-2xy+yz & z
\end{bmatrix},
\]
it follows that  $I_{\epsilon_1} = \langle 2xy-yz, z\rangle$. First, we calculate the Bourbaki degree by the definition.
In the quotient $A = R/I_{\epsilon_1}$, we have $z = 0$ and $2xy=yz$, which implies that there are no mixed terms. Therefore, if $g(x,y,z) \in A$, then 
 \[g(x,y,z) = a_0 + \sum_{n \geq 1}\left(a_n x^n + b_n y^n\right)\] 
 and hence $\mathrm{HP_A}(t) = 2$ for $t \geq 1$, and we conclude that $\mathrm{Bour}(F) = \deg (A) = 2$.
		
Using the local Bourbaki formula, we determine the zeros of the ideal $I_{\epsilon_1}$, which can be verified to be $P_1 = (1:0:0)$ and $P_2 = (0:1:0)$.
In the case of the point $P_1$, the dehomogenized Bourbaki ideal is $I_{\epsilon_1}(1,y,z) = \langle 2y-yz, z\rangle = \langle y,z \rangle$, and in the case of $P_2$, we have $I_{\epsilon_1}(x,1,z) = \langle 2x-z,z\rangle = \langle x,z\rangle$.
In each case, the local Bourbaki degree is $1$. Finally, by the local formula for the Bourbaki degree, one has $\mathrm{Bour}(F) = 1+1 = 2$. We observe that $|V(I_\epsilon)|=2 < 3=|V(J_F)|$.
	\end{exa} 
	
In the following result, we present a family of nearly-free curves that can have arbitrarily large degree.

\begin{prop}\label{prop1}
		Let $F: y^az^{b-a} - x^b=0$, with $b>a>1$. Then $\tau(F) = (b-1)(b-2)$ and $\mathrm{Bour}(F)=1$.
\end{prop}

\begin{proof}
First, we determine the singular points of $F$. We have \[J_F = \langle -bx^{b-1}, ay^{a-1}z^{b-a}, (b-a)y^az^{b-a-1} \rangle.\]
		
Solving the system $F=F_x=F_y=F_z=0$, we obtain two possibilities: 
		
$(1)$ if $b-a>1$, then $\mathrm{Sing}(F)=\{ (0:1:0),(0:0:1) \}$,
		
$(2)$ if $b-a=1$, then $\mathrm{Sing}(F) = \{(0:0:1)\}$. 
		
We treat case $(1)$, since case $(2)$ follows in a completely analogous way. Let $P_1 = (0:1:0)$ and $P_2 = (0:0:1)$. We compute $\tau_{P_1}(F)$.
		
Setting $y=1$, we obtain $F(x,1,z) = z^{b-a} - x^b$, $F_x(x,1,z) = -bx^{b-1}$, and $F_z(x,1,z) = (b-a)z^{ b-a-1}$. Since $P_1 \in F$ is the unique singular point in the chart $y \neq 0$,  
\[\tau_{P_1}(F) := \mathrm{dim}_k \left(\frac{k[x,z]}{\langle f, f_x, f_z\rangle}\right)_{\langle x,z \rangle} = \mathrm{dim}_k \frac{k[x,z]}{\langle f, f_x, f_z\rangle},\] 
where $f = F(x,1,z)$, $f_x = F_x(x,1,z)$, and $f_z = F_z(x,1,z)$. In the quotient, we have $x^i = 0$, for all $i \geq b-1$ and $z^{j} =0$, for all $j \geq b-a-1$.
Thus, $\tau_{P_1}(F) = (b-1)(b-a-1)$. Analogously, we conclude that $\tau_{P_2}(F) = (b-1)(a-1)$. Therefore,
\[ \tau(F) = \tau_{P_1}(F) + \tau_{P_2}(F) = (b-1)\cdot(b-a-1) + (b-1)\cdot (a-1) = (b-1)\cdot (b-2).\]
Now, it remains to show that $\mathrm{Bour}(F)=1$. We prove this using the local Bourbaki formula. By \cite{Singular}, $\mathrm{Syz}(J_F) = \langle \epsilon_1, \epsilon_2, \epsilon_3\rangle$, with $ \deg \epsilon_1=1$ and $\deg \epsilon_2 = \deg\epsilon_3 = b-1$, and these generators satisfy
\[ (-x^{b-1})\cdot \epsilon_1 + y \cdot \epsilon_2 + (-z)\cdot \epsilon_3=0.\]

Using the same argument used in the last example, we have $I_{\epsilon_1} = \langle y,z\rangle $, so that  
$V(I_{\epsilon_1}) = \{(1:0:0)\}$. Thus, since there is a single point, the local Bourbaki formula guaranties that
\[ \mathrm{Bour}(F) = \dim_k \left( \dfrac{k[y,z]}{\langle y,z \rangle}\right)_{\langle y,z \rangle} =1,
\]
and all the assertions follow.		
\end{proof}

In the last two results of this section, we apply the local Bourbaki formula to show that, for every $n \in \mathbb{N}$, there exists an irreducible projective plane curve $F$ such that $\mathrm{Bour}(F)=n$.

\begin{prop}\label{prop3}
		Let $F: x^{2k+1}z +x^{k+1}y^{k+1}+y^{2k+1}z=0 $ be a curve in $\mathbb{P}^2$, with $k\geq2$. Then $\mathrm{Bour}(F) =k+4$.  
\end{prop}

\begin{proof}
By \cite{Singular}, we conclude that $\mathrm{Syz}(J_F) = \langle \epsilon_1, \epsilon_2, \epsilon_3, \epsilon_4, \epsilon_5 \rangle$, where $\deg \epsilon_{i} = k+2$ for $i = 1,2,3,4$ and $\deg \epsilon_5 = 2k$. Moreover, the generators satisfy the following relations
		\begin{align*}
			\begin{cases}
				(2k+1)z\cdot \epsilon_1 + (-y) \cdot \epsilon_3 + (k+1)x \cdot \epsilon_4 = 0\\
				-(k+1)y \cdot \epsilon_1 + x\cdot \epsilon_2 + (-(2k+1))z\cdot \epsilon_4 = 0\\
				(-y^{k-1})\cdot \epsilon_2 + x^{k-1} \epsilon_3 + (4k+2)z\cdot \epsilon_5=0
			\end{cases}.
		\end{align*}
		We choose $\epsilon_1$ as a minimal generating syzygy. Thus, $M_{\epsilon_1} = \langle \delta_2, \delta_3, \delta_4, \delta_5 \rangle$, where $\delta_i \equiv \epsilon_{i} (\bmod \hspace{0.1cm} \epsilon_1)$ for $i=2,3,4,5$. Hence, in the quotient the following relations hold
		\begin{align*}
			\begin{cases}
				(-y) \cdot \delta_3 + (k+1)x \cdot \delta_4 = 0\\
				+ x\cdot \delta_2 + (-(2k+1))z\cdot \delta_4 = 0\\
				(-y^{k-1})\cdot \delta_2 + x^{k-1} \delta_3 + (4k+2)z\cdot \delta_5=0 
			\end{cases}.
		\end{align*}
    
	Then $I_{\epsilon_1}  = \langle yz^2, x^2z, xyz, (2k+1)y^kz-(k+1)x^{k+1} \rangle$. It is straightforward to verify that $V(I_{\epsilon_1}) = \{ P_1, P_2\}$, where $P_1 = (0:1:0)$ and $P_2 = (0:0:1)$. We shall compute the local Bourbaki degree at each of these points. In the case of the point $P_1$, dehomogenizing with respect to $y$, we obtain $I_{\epsilon_1} = \langle z^2, x^2z, xz, (2k+1)z-(k+1)x^{ k+1}\rangle$, and a simple calculation shows that $\mathrm{Bour}_{ P_1} (F,  \epsilon_1)= k+2$. Similarly, one shows that $\mathrm{Bour}_{P_2}(F, \epsilon_1) = 2$. 
		
Therefore, $\mathrm{Bour} (F) = \mathrm{Bour}_{ P_1} (F,\epsilon_1)+\mathrm{Bour}_{ P_2} (F,\epsilon_1) = k+2+2=k+4$. This proves the assertion.\end{proof}
	
	\begin{thm}\label{prop-exis-grau-varia}
		For every $n \in \mathbb{N}$, there exists an irreducible plane curve $F$ in $\mathbb{P}^2$ such that $\mathrm{Bour}(F) = n$. 
	\end{thm}
	\begin{proof}
		By Proposition \ref{prop3}, it suffices to present curves for the cases $n=0,1,2,3,4,5$.
Consider the following
\begin{center}
    \begin{tabular}{c|c}
			\text{Curve equation}& \text{Bourbaki degree } \vspace{0.1cm}\\
			
			\hline 
			\vspace{0.1cm}
			$F_a: x^{2a+1} + x^a y^{a+1} + y^{2a}z = 0$, $a \geq 2$& $0$ \\
			\hline
			\vspace{0.1cm}
			$F_{a,b}: y^a z^{b-a}- x^b=0$, $b>a>1$, $\operatorname{gcd}(a,b)=1$ & $1$ \\
			\hline
			\vspace{0.1cm}
			$F_2: x^3y + x^4 + y^4 + xy^2z= 0$& $2$ \\
			\hline
			\vspace{0.1cm}
			$F_3: x^3 + x^2z -y^2z =0$ & $3$ \\
			\hline
			\vspace{0.1cm}
			$F_4: x^3y+x^2y^2 +y^4 -x^4+y^2z^2=0$& $4$ \\
			\hline
            \vspace{0.1cm}
            $F_5: x^5 + x^4y+ x^3z^2 +y^2z^3=0$ & $5$ \\
            \hline
            
		\end{tabular}
\end{center}
   In Example \ref{ex1}, we have proved that the Bourbaki degree of $F_3$ is $3$. Moreover, by Proposition \ref{prop1}, one has $\mathrm{Bour}(F_{a,b}) = 1$. The computation of the Bourbaki degrees of $F_2$ and $F_4$ can be found in \cite{Daniel}, while for $F_5$, see \cite{JNS2024}*{Ex. 2.8}. Thus, it remains to show that $\mathrm{Bour}(F_a) = 0$, for all $a\geq 2$.

Computing the Jacobian ideal of $F_a$, we obtain
    \[ 
    J_{F_a} = \langle (2a+1)x^{2a} + ax^{a-1} y^{a+1}, (a+1)x^a y^a+2a y^{2a-1}z, y^{2a} \rangle.
    \]
 Now, consider  \[ \theta_1 = (0, -y^a, (a+1)x^{a} + 2ay^{a-1}z)\] and \[ \theta_2 = (-(a+1)^2y^a, (a+1)(2a+1)x^a - 2a(2a+1)y^{a-1}z, a(a+1)^2x^{a-1}y + 4a^2(2a+1)y^{a-2}z^2).\]
A simple calculation shows that $\theta_1, \theta_2 \in \mathrm{Syz}(J_F)$. Furthermore, both syzygies are linearly independent and $\deg \theta_1 + \deg \theta_2 +1 = \deg F_a$. Then, by Saito's criterion (see, for instance, \cite{Saito}*{p. 273}),  we conclude that $F_a$ is a free curve, for all $a \geq 2$. Therefore, $\mathrm{Bour}(F_a) = 0$. Furthermore, it is straightforward to verify that all of the curves are irreducible. \end{proof}

\begin{rem}
    We note that the family $F_a$ described in the proof of the above theorem is a family of irreducible free curves of degree $2a+1$, for every $a\geq 2$. The problem of finding irreducible free curves for higher degree curves is pointed out in \cite{simis06} and studied in \cite{nanduri15} where the author provides a family of irreducible free curves of any degree $d \geq 5$. 
\end{rem}


\section{Classification of the Bourbaki Degree via Free Resolutions}

We now study the inverse problem: for a prescribed integer $n>0$, which minimal graded free resolutions of $R/J_F$ are compatible with $\operatorname{Bour}(F)=n$? Since the Bourbaki ideal of a non-free curve is a perfect ideal of height two, the problem is controlled by its Hilbert-Burch resolution.

\subsection{Numerical restrictions on Hilbert-Burch resolutions}
\label{subsec:classification-numerical}

Let $I\subset R=k[x,y,z]$ be a homogeneous perfect ideal of height $2$ with minimal graded free resolution
\[ 0\longrightarrow\bigoplus_{j=1}^{n-1}R(-a_j)
\xrightarrow{\ \varphi\ } \bigoplus_{i=1}^{n}R(-b_i)
\longrightarrow I\longrightarrow0, \tag{4} \]
where $b_1\le\cdots\le b_n, \text{ and } a_1\le\cdots\le a_{n-1}. $

We first collect the numerical facts that will be used in the classification theorem.

\begin{lem}\label{lemma:igualdades-coeficientes-grau}
With the above notation one has
\begin{enumerate}
    \item[(I)]    $\displaystyle\sum_{i=1}^{n}b_i=\displaystyle\sum_{j=1}^{n-1}a_j.$
    
   \item[(II)]    $\deg(R/I)=\dfrac{1}{2}\left(\displaystyle\sum_{j=1}^{n-1}a_j^2-\displaystyle\sum_{i=1}^{n}b_i^2\right).$ 
   
 \item[(III)]     For every $j=1,\ldots,n-1$, one has $a_j>b_{j+1}.$
\end{enumerate}
\end{lem}

\begin{proof}
    See \cite{Eisenbud2005}*{Prop. 3.8}.
\end{proof}

The next proposition follows immediately from previous lemma. 

\begin{prop}\label{prop:nova-formula}
Let  $F$ be a reduced plane projective curve of degree $d+1$ and $\epsilon \in \operatorname{Syz}(J_F)$ a minimal generator. Let $I_\epsilon$ be the Bourbaki ideal associated to $\epsilon$. Suppose that $I_\epsilon$ has a minimal free resolution as in (4). Then 
\[ \operatorname{Bour}(F) = \frac{1}{2}\left(\sum_{j=1}^{n-1}a_j^2-\sum_{i=1}^{n}b_i^2\right).
 \]
\end{prop}

\begin{lem}\label{lem:number-generators}
If $n$ is the number of minimal generators of $I$, then
\[ \deg(R/I)\ge\frac{n(n-1)}2. \tag{IV} \]
\end{lem}

\begin{proof}
Set $c_j=a_j-b_{j+1}$. By Lemma  \ref{lemma:igualdades-coeficientes-grau} (III), $c_j\ge1$. From Lemma \ref{lemma:igualdades-coeficientes-grau} (I),
$ b_1=\sum_{j=1}^{n-1}c_j. $
By Lemma  \ref{lemma:igualdades-coeficientes-grau}  (II),
\[ 2\deg(R/I) =\sum_{j=1}^{n-1}(b_{j+1}+c_j)^2
-b_1^2-\sum_{j=1}^{n-1}b_{j+1}^2 =2\sum_{j=1}^{n-1}b_{j+1}c_j +\sum_{j=1}^{n-1}c_j^2-b_1^2.\]
Since $b_{j+1}\ge b_1$,
\[2\deg(R/I)\ge b_1^2+\sum_{j=1}^{n-1}c_j^2 =\left(\sum_{j=1}^{n-1}c_j\right)^2+\sum_{j=1}^{n-1}c_j^2.
\]
Finally, $c_j\ge1$ gives the lower bound $(n-1)^2+(n-1)=n(n-1)$. This proves the desired lower bound.
\end{proof}

\subsection{The general classification theorem}\label{subsec:general-classification}  The following theorem is a generalization of \cite{JNS2024}*{Prop. 2.19}.

\begin{thm}\label{prop:classification-general}
Let $n\geq 1$ be an integer and let $F$ be a singular reduced projective plane curve of degree $d+1$, with
$e=\operatorname{indeg}(\operatorname{Syz}(J_F))$. Assume that $F$ is not free. Then $\operatorname{Bour}(F)=n$ if and only if there exist an integer $m\geq 2$, positive integers $c_1,\ldots,c_{m-1}$, and nonnegative integers $u_1,\ldots,u_{m-1}$ satisfying
	\[ 2\leq m\leq \left\lfloor\frac{1+\sqrt{1+8n}}{2}\right\rfloor,\qquad 0\leq u_1\leq\cdots\leq u_{m-1},\qquad u_1+c_1\leq\cdots\leq u_{m-1}+c_{m-1},
	\]
	and
	\[ n=\frac{1}{2}\left[\left(\sum_{j=1}^{m-1} c_j\right) ^2+\sum_{j=1}^{m-1}c_j^2\right]
	+\sum_{j=1}^{m-1}u_jc_j,
	\]
such that, setting $s=\sum_{j=1}^{m-1}c_j$, the ring $R/J_F$ has a minimal free resolution of the form
\[
	\begin{aligned}
		0\longrightarrow&
		\bigoplus_{j=1}^{m-1}R\bigl(-(2d-e+s+u_j+c_j)\bigr)\\
		\longrightarrow&
		R(-(d+e))\oplus R(-(2d-e+s))
		\oplus\bigoplus_{j=1}^{m-1}R\bigl(-(2d-e+s+u_j)\bigr)\\
		\longrightarrow& R^3(-d)\longrightarrow R\longrightarrow R/J_F\longrightarrow 0.
	\end{aligned}
	\]
In this case, necessarily $e \leq (d+s)/2$.
	
Moreover, for each fixed $n$, there are only finitely many possibilities for the integers $m, c_j, u_j$, $j=1,\ldots,m-1$ satisfying the conditions above. Consequently, there is a finite procedure for listing all possible  minimal free resolutions of $R/J_F$ for a curve satisfying $\operatorname{Bour}(F)=n$.
\end{thm}

\begin{proof}
Assume first that $\operatorname{Bour}(F)=n$. Let $\epsilon$ be a minimal syzygy of $J_F$ of degree $e$. Since $F$ is not free, the Bourbaki ideal $I_\epsilon$ is a proper perfect ideal of height $2$. Therefore, its minimal free resolution is of the form
\[ 0\longrightarrow \bigoplus_{j=1}^{m-1}R(-a_j)
\longrightarrow \bigoplus_{i=1}^{m}R(-b_i)\longrightarrow I_\epsilon\longrightarrow 0,\]
where we may order the degrees so that $a_1\leq\cdots\leq a_{m-1}$ and $b_1\leq\cdots\leq b_m$.
	
By Lemma \ref{lemma:igualdades-coeficientes-grau} (I) and (II), we have
\[\sum_{j=1}^{m-1}a_j=\sum_{i=1}^{m}b_i \quad \text{and} \quad \deg(R/I_\epsilon) =\frac{1}{2}\left(\sum_{j=1}^{m-1}a_j^2-\sum_{i=1}^{m}b_i^2\right). \]
Moreover, by the degree properties of a minimal resolution of this type, we have $a_j>b_{j+1}$ for every $j=1,\ldots,m-1$.
	
Define $c_j=a_j-b_{j+1}$. Then $c_j\geq 1$ for every $j$. From the equality between the sums of the degrees, we obtain
\[ \sum_{j=1}^{m-1}(b_{j+1}+c_j) =\sum_{i=1}^{m}b_i,\]
and hence $b_1=\sum_{j=1}^{m-1}c_j$. Setting $s=\sum_{j=1}^{m-1}c_j$, we have $b_1=s$.
	
Now define $u_j=b_{j+1}-b_1$ for $j=1,\ldots,m-1$. Since $b_1\leq b_2\leq\cdots\leq b_m$, it follows that
$0\leq u_1\leq\cdots\leq u_{m-1}$. Moreover,
$b_{j+1}=s+u_j$ and $a_j=b_{j+1}+c_j=s+u_j+c_j$. Since $a_1\leq\cdots\leq a_{m-1}$, we also obtain
$u_1+c_1\leq\cdots\leq u_{m-1}+c_{m-1}$.
	
Consequently, the minimal free resolution of $I_\epsilon$ can be written as
\[ 0\longrightarrow \bigoplus_{j=1}^{m-1}R(-(s+u_j+c_j))
\longrightarrow R(-s)\oplus\bigoplus_{j=1}^{m-1}R(-(s+u_j)) \longrightarrow I_\epsilon\longrightarrow 0.\]
Since $\operatorname{Bour}(F)=\deg(R/I_\epsilon)=n$, the degree formula yields
\[ 2n =\sum_{j=1}^{m-1}(s+u_j+c_j)^2-s^2-\sum_{j=1}^{m-1}(s+u_j)^2 =s^2+2\sum_{j=1}^{m-1}u_jc_j+\sum_{j=1}^{m-1}c_j^2,\]
since $s=\sum_{j=1}^{m-1}c_j$. Therefore,
\[ n=\frac{1}{2}\left[s^2+\sum_{j=1}^{m-1}c_j^2\right] +\sum_{j=1}^{m-1}u_jc_j =\frac{1}{2}\left[\left(\sum_{j=1}^{m-1}c_j\right)^2 +\sum_{j=1}^{m-1}c_j^2\right] +\sum_{j=1}^{m-1}u_jc_j.\]
By (IV), $\deg(R/I_\epsilon) \geq (m^2 -m)/2$, and therefore $m^2-m-2n\leq 0$. Hence,
	\[ m\leq\frac{1+\sqrt{1+8n}}{2}.\]
Since $m$ is an integer, it follows that $m\leq\left\lfloor(1+\sqrt{1+8n})/2\right\rfloor$.
	
By \cite{JNS2024}*{Thm. 2.1(c)}, the above resolution of $I_\epsilon$ determines the minimal free resolution of $\operatorname{Syz}(J_F)$. Shifting this resolution by $-d$ and completing it with $R^3(-d)\longrightarrow R$, we obtain precisely
\[
	\begin{aligned}
		0\longrightarrow&
		\bigoplus_{j=1}^{m-1}R\bigl(-(2d-e+s+u_j+c_j)\bigr)\\
		\longrightarrow&
		R(-(d+e))\oplus R(-(2d-e+s))
		\oplus\bigoplus_{j=1}^{m-1}R\bigl(-(2d-e+s+u_j)\bigr)\\
		\longrightarrow& R^3(-d)\longrightarrow R\longrightarrow R/J_F\longrightarrow 0.
	\end{aligned}
\]
	
On the other hand, before shifting by $-d$, the degrees of the minimal generators of $\operatorname{Syz}(J_F)$ are
$e$, $d-e+s$, and $d-e+s+u_j$, for $j=1,\ldots,m-1$. Since
$e=\operatorname{indeg}(\operatorname{Syz}(J_F))$ and $u_j\geq 0$, necessarily $e\leq d-e+s$. Therefore, $e \leq (d+s)/2$.

Conversely, suppose that $R/J_F$ has a minimal free resolution of the form given in the statement. Removing the terms $R^3(-d)\longrightarrow R$ and shifting the remaining modules by $d$, we obtain a minimal free resolution of $\operatorname{Syz}(J_F)$.
	
Applying again \cite{JNS2024}*{Thm. 2.1(c)}, the Bourbaki ideal $I_\epsilon$ associated with the minimal generator of degree $e$ has a minimal free resolution
	\[
	0\longrightarrow
	\bigoplus_{j=1}^{m-1}R(-(s+u_j+c_j))
	\longrightarrow
	R(-s)\oplus\bigoplus_{j=1}^{m-1}R(-(s+u_j))
	\longrightarrow I_\epsilon\longrightarrow 0.
	\]
Hence, by the degree formula,
	\[
		2\deg(R/I_\epsilon)
		=\sum_{j=1}^{m-1}(s+u_j+c_j)^2-s^2-\sum_{j=1}^{m-1}(s+u_j)^2
		=s^2+\sum_{j=1}^{m-1}c_j^2+2\sum_{j=1}^{m-1}u_jc_j.
	\]
By the assumptions on the integers $c_j$ and $u_j$, the right-hand side is equal to $2n$. Therefore, $\deg(R/I_\epsilon)=n$, that is, $\operatorname{Bour}(F)=n$.
	
It remains to show that, for fixed $n$, there are only finitely many possibilities for the integers involved. From the equality
	\[ n=\frac{1}{2}\left(s^2+\sum_{j=1}^{m-1}c_j^2\right)
	+\sum_{j=1}^{m-1}u_jc_j \]
and the conditions $c_j\geq 1$ and $u_j\geq 0$, we have $\sum_{j=1}^{m-1}c_j^2\geq\sum_{j=1}^{m-1}c_j=s$. Therefore,  $n\geq\frac{s^2+s}{2},$
implying that
	\[ s\leq\left\lfloor\frac{-1+\sqrt{1+8n}}{2}\right\rfloor. \]
Thus, there are only finitely many possibilities for $s$, and, for each fixed $s$, only finitely many possibilities for writing $s=c_1+\cdots+c_{m-1}$ with $c_j\geq 1$.
	
Once $m$ and the $c_j$ are fixed, we have $\sum_{j=1}^{m-1}u_jc_j\leq n$. Since all terms are nonnegative, it follows that $u_jc_j\leq n$ for each $j$. Since $c_j\geq 1$, we obtain $0\leq u_j\leq n$. Hence, there are also only finitely many possibilities for $u_1,\ldots,u_{m-1}$. This concludes the proof.
\end{proof}

To illustrate how the free resolutions in Theorem \ref{prop:classification-general} behave as the Bourbaki degree increases, we conclude this section by considering the cases $n=3,4$. Recall that the cases $n=0,1,2$ are treated in \cite{JNS2024}.

\begin{prop}
Let $F$ be a singular curve in $\mathbb P^2$ of degree $d+1$ and let $e=\operatorname{indeg}(\operatorname{Syz}(J_F))$. Assume that $F$ is not free. Then $\operatorname{Bour}(F)=3$ if and only if one of the following alternatives holds:
\begin{enumerate}
\item $R/J_F$ has a minimal free resolution of the form
\[
\begin{aligned}
0 &\longrightarrow R(-(d+(d-e+4)))\\
&\longrightarrow R(-(d+e))\oplus R(-(d+(d-e+1)))\oplus R(-(d+(d-e+3)))\\
&\longrightarrow R^3(-d)\longrightarrow R\longrightarrow R/J_F\longrightarrow0,
\end{aligned}
\]
with $e\leq(d+1)/2$;
\item $R/J_F$ has a minimal free resolution of the form
\[
\begin{aligned}
0 &\longrightarrow R^2(-(d+(d-e+3)))\\
&\longrightarrow R(-(d+e))\oplus R^3(-(d+(d-e+2)))\\
&\longrightarrow R^3(-d)\longrightarrow R\longrightarrow R/J_F\longrightarrow0,
\end{aligned}
\]
with $e\leq(d+2)/2$.
\end{enumerate}
\end{prop}

\begin{proof}
Apply Theorem \ref{prop:classification-general} with $n=3$. The numerical conditions leave exactly the two Hilbert-Burch resolutions
\[ 0\to R(-4)\to R(-1)\oplus R(-3)\to I_\epsilon\to0 \]
and
\[ 0\to R^2(-3)\to R^3(-2)\to I_\epsilon\to0. \]
The corresponding resolutions of $R/J_F$ follow from
\cite{JNS2024}*{Thm. 2.1(c)}, and the restrictions on $e$ come from the minimality of $e$. Conversely, Proposition \ref{prop:nova-formula} gives degree $3$ for both Bourbaki ideals.
\end{proof}

\begin{prop}
Let $F$ be a singular curve in $\mathbb P^2$ of degree $d+1$ and let $e=\operatorname{indeg}(\operatorname{Syz}(J_F))$. Assume that $F$ is not free. Then $\operatorname{Bour}(F)=4$ if and only if one of the following alternatives holds:
\begin{enumerate}
\item
\[
\begin{aligned}
0 &\longrightarrow R(-(d+(d-e+4)))\\
&\longrightarrow R(-(d+e))\oplus R^2(-(d+(d-e+2)))\\
&\longrightarrow R^3(-d)\longrightarrow R\longrightarrow R/J_F\longrightarrow0,
\end{aligned}
\]
with $e\leq(d+2)/2$;
\item
\[
\begin{aligned}
0 &\longrightarrow R(-(d+(d-e+5)))\\
&\longrightarrow R(-(d+e))\oplus R(-(d+(d-e+1)))\oplus R(-(d+(d-e+4)))\\
&\longrightarrow R^3(-d)\longrightarrow R\longrightarrow R/J_F\longrightarrow0,
\end{aligned}
\]
with $e\leq(d+1)/2$;
\item
\[
\begin{aligned}
0 &\longrightarrow R(-(d+(d-e+3)))\oplus R(-(d+(d-e+4)))\\
&\longrightarrow R(-(d+e))\oplus R^2(-(d+(d-e+2)))\oplus R(-(d+(d-e+3)))\\
&\longrightarrow R^3(-d)\longrightarrow R\longrightarrow R/J_F\longrightarrow0,
\end{aligned}
\]
with $e\leq(d+2)/2$.
\end{enumerate}
\end{prop}

\begin{proof}
Apply Theorem \ref{prop:classification-general} with $n=4$. The numerical conditions give precisely the three Bourbaki resolutions
\[
0\to R(-4)\to R^2(-2)\to I_\epsilon\to0,
\qquad
0\to R(-5)\to R(-1)\oplus R(-4)\to I_\epsilon\to0,
\]
and
\[
0\to R(-3)\oplus R(-4)\to R^2(-2)\oplus R(-3)\to I_\epsilon\to0.
\]
Using \cite{JNS2024}*{Thm. 2.1(c)} gives the three displayed resolutions of $R/J_F$, while Proposition \ref{prop:nova-formula} gives degree $4$ in each case.
\end{proof}


\section{Improved Bounds for the Bourbaki Degree} \label{sec:bounds}
Let $\epsilon=(a,b,c)\in\operatorname{Syz}(J_F)$ be a minimal syzygy of degree $e\ge1$, and set
$H_\epsilon =\langle a,b,c\rangle\subset R.$  In \cite{JNS2024}*{Thm. 2.10}, the authors show that $V(I_{\epsilon}) \subseteq V(H_\epsilon)$. In what follows, we use the notation  
\[ \ell_{\epsilon} := |V(H_\epsilon) \setminus V(I_{\epsilon})|.   \]
The main idea of this section is to compare $\deg R/H_\epsilon$ and $\deg R/I_\epsilon$.

\begin{lem}\label{lem:height-H}
If $\operatorname{Bour}(F)>0$, then $\operatorname{ht}(H_\epsilon)=2$.
\end{lem}

\begin{proof}
The coordinates of a minimal syzygy have no nonconstant common factor. Indeed, if a homogeneous polynomial $g$ divided all nonzero entries of $\epsilon$, then $\epsilon=g\epsilon'$ and, because $R$ is a domain, $\epsilon'$ would be a syzygy of strictly smaller degree, contradicting the minimality of $e$. Hence $\operatorname{ht}(H_\epsilon)\ne1$. On the other hand, $V(I_\epsilon)\subseteq V(H_\epsilon)$ by \cite{JNS2024}*{Thm. 2.10}, and when $\operatorname{Bour}(F)>0$ the Bourbaki ideal is proper of height two. Therefore $\operatorname{ht}(H_\epsilon)=2$.
\end{proof}

In particular, $V(H_\epsilon)$ is finite. Since $H_\epsilon$ is generated by forms of degree $e$, Bézout gives $\deg(R/H_\epsilon)\le e^2$. Moreover, \cite{JNS2024}*{Thm. 2.10} gives
\[ V(I_\epsilon)\subseteq V(H_\epsilon), \qquad \ell((R/I_\epsilon)_P)\le\ell((R/H_\epsilon)_P) \quad(P\in V(I_\epsilon)). \]

\begin{thm}\label{prop:basic-bound}
With the  above notation,
\[ \operatorname{Bour}(F) \le e^2-\ell_{\epsilon}. \]
In particular, if $\operatorname{Bour}(F)=e^2$, then $V(H_\epsilon)=V(I_\epsilon)$.
\end{thm}

\begin{proof}
Using the local decomposition of the degree of $R/H_\epsilon$,
\[ \deg(R/H_\epsilon) =\sum_{P\in V(I_\epsilon)}\ell((R/H_\epsilon)_P) +\sum_{P\in V(H_\epsilon)\setminus V(I_\epsilon)}\ell((R/H_\epsilon)_P). \]
The local comparison gives
\[\deg(R/H_\epsilon)\ge\operatorname{Bour}(F)+\sum_{P\in V(H_\epsilon)\setminus V(I_\epsilon)}\ell((R/H_\epsilon)_P). \]
Every term in the second sum is at least one, and hence
\[ \operatorname{Bour}(F) \le\deg(R/H_\epsilon)-\ell_{\epsilon}
\le e^2-\ell_{\epsilon}, \]
as we wanted to prove.
\end{proof}

\begin{rem}
The converse of Theorem \ref{prop:basic-bound} does not hold, even if $V(H_\epsilon)=V(I_\epsilon)$ for every nonzero syzygy $\epsilon$ of minimal degree.
Indeed, consider the reduced quartic curve $F=(x^2+y^2)^2-4xy^2z$. In this case $e=\operatorname{indeg}(\operatorname{Syz}(J_F))=2$, and the degree-$2$ part of $\operatorname{Syz}(J_F)$ is generated by
\[ \epsilon_1= (-y^2+3xz-2z^2,\;xy-yz,\;-y^2) \quad\text{and}\quad \epsilon_2= (x^2+xz-2z^2,\;-yz,\;-y^2). \]
For these two generators, the corresponding Bourbaki ideals are $I_{\epsilon_1}=\langle y^2-3xz+2z^2,y\rangle$ and $I_{\epsilon_2}=\langle x^2+xz-2z^2,y\rangle$. In particular,
\[ V(H_{\epsilon_1})  =\{(1:0:0),(2:0:3)\} =V(I_{\epsilon_1}) \ \text{and} \ V(H_{\epsilon_2}) =\{(1:0:1),(-2:0:1)\} =V(I_{\epsilon_2}).\]
Since $I_{\epsilon_1}$ is a complete intersection generated by forms of degrees $1$ and $2$, $\operatorname{Bour}(F)=\deg(R/I_{\epsilon_1})=2$. Therefore
\[ \operatorname{Bour}(F)=2<4=e^2. \]
Thus the equality $V(H_\epsilon)=V(I_\epsilon)$, even when it holds for both minimal-degree syzygies, does not imply $\operatorname{Bour}(F)=e^2$.
\end{rem}

From now on, we assume that  $a,b,c$ are nonzero and linearly independent over $k$. Thus $H_\epsilon$ has exactly three minimal generators, all of degree $e$.

\begin{thm}\label{prop:perfect-bound}
Assume that $H_\epsilon$ is perfect. Then there exist integers $r,s\ge1$ satisfying $r+s=e$ such that
\[\operatorname{Bour}(F)\le e^2-rs.\]
Consequently,
\[\operatorname{Bour}(F)\le e^2-e+1.\]
\end{thm}

\begin{proof}
Hilbert-Burch yields
\[0\longrightarrow R(-(e+r))\oplus R(-(e+s))\longrightarrow R^3(-e)\longrightarrow H_\epsilon\longrightarrow0.\]
Minimality implies $r,s\ge1$. By Lemma  \ref{lemma:igualdades-coeficientes-grau} (I), one has  $3e=(e+r)+(e+s)$, and thus $r+s=e$. Lemma \ref{lemma:igualdades-coeficientes-grau}    (II) gives
\[ \deg(R/H_\epsilon) =\frac12\bigl((e+r)^2+(e+s)^2-3e^2\bigr) =e^2-rs. \]
Since $\operatorname{Bour}(F)\le\deg(R/H_\epsilon)$ and
\[ rs-(e-1)=rs-r-s+1=(r-1)(s-1)\ge0, \]
we obtain the desired bound.
\end{proof}

\begin{rem}
The extremal value $e^2-e+1$ in Theorem \ref{prop:perfect-bound} is consistent with
\cite{JardimMonteiroNejadRamos2026}*{Rem. 3.9}, where the value
$e^2-e+1$ is realized for perfect three-generated $e$-equigenerated
ideals whose Hilbert-Burch syzygy degrees are $1$ and $e-1$.
Theorem \ref{prop:perfect-bound} gives the finer bound
$\operatorname{Bour}(F)\leq e^2-rs$, where $r+s=e$. Thus the value
$e^2-e+1$ corresponds to the extreme case $r=1$ or $s=1$. If $r,s>1$, then
\[
e^2-rs<e^2-e+1.
\]
As the following example shows, this sharper bound can in fact be attained.
\end{rem}

\begin{exa}
 Consider the reduced curve
\[ F=(x^4+y^4+z^4)(x^4+2y^4+3z^4). \]
Using \textsc{Singular}, we obtain the minimal syzygy $\epsilon=(y^3z^3,-2x^3z^3,x^3y^3)$ of degree $e=6$. Set $H_\epsilon=\langle y^3z^3,-2x^3z^3,x^3y^3\rangle$. Its three generators are linearly independent, and \textsc{Singular} gives the minimal graded free resolution
\[ 0\longrightarrow R(-9)^2\longrightarrow R(-6)^3\longrightarrow H_\epsilon\longrightarrow0. \]
In particular, $H_\epsilon$ is perfect. Comparing this resolution with the one in Theorem \ref{prop:perfect-bound}, we obtain $r=s=3$. Hence $\deg(R/H_\epsilon)=e^2-rs=36-9=27$.

For the Bourbaki ideal $I_\epsilon$, \textsc{Singular} also gives the minimal graded free resolution \[0\longrightarrow R(-9)^2\longrightarrow R(-6)^3\longrightarrow I_\epsilon\longrightarrow0.\]
Therefore, $\operatorname{Bour}(F)=\deg(R/I_\epsilon)=27$. Thus, in this example, the first bound in Theorem \ref{prop:perfect-bound} is attained:
\[\operatorname{Bour}(F)=27=e^2-rs.\]
Since $e^2-rs=27<31=e^2-e+1$, the bound involving $rs$ is strictly sharper.
\end{exa}

The next result follows from \cite{JardimMonteiroNejadRamos2026}*{Thm. 3.8}. We present a different proof, which was independently obtained.

\begin{prop}\label{prop:nonperfect-bound}
Let $R=k[x,y,z]$. If $H_\epsilon$ is not perfect, then
\[\deg(R/H_\epsilon)\le e^2-e.\]
Hence, $\operatorname{Bour}(F)\le e^2-e$.
\end{prop}

\begin{proof}
Since $H_\epsilon$ has height two but is not perfect, its minimal free resolution has length two:
\[ 0\longrightarrow\bigoplus_{i=1}^{p}R(-s_i)\longrightarrow\bigoplus_{j=1}^{p+2}R(-d_j)\longrightarrow R^3(-e)\longrightarrow H_\epsilon\longrightarrow0,
\]
where $d_1\le\cdots\le d_{p+2}$ and $s_1\le\cdots\le s_p$. The equality of ranks shows that there are $p+2$ first syzygies.
From the Hilbert polynomial of $R/H_\epsilon$, we obtain
\[ \sum_{i=1}^{p}s_i=\sum_{j=1}^{p+2}d_j-3e \tag{5} \]
and
\[ 2\deg(R/H_\epsilon)=\sum_{j=1}^{p+2}d_j^2-\sum_{i=1}^{p}s_i^2-3e^2. \tag{6} \]
Minimality and the ordering of the shifts imply $s_i>d_{i+2}$, for all $i=1,\ldots,p$. 
Indeed, if $s_i\le d_{i+2}$ for some $i$, then degree considerations force the first $i$ relations to involve only the first $i+1$ syzygies. Passing to the fraction field would then force these $i+1$ syzygies to span a one-dimensional space, contradicting the minimality of the generating set.

Write $u_i=s_i-d_{i+2}\ge1$ and $u=\sum_{i=1}^{p}u_i$. Using (5), we obtain
\[ u=d_1+d_2-3e. \tag{7} \]
Substituting $s_i=d_{i+2}+u_i$ in (6) gives
\[ 2\deg(R/H_\epsilon)=d_1^2+d_2^2-3e^2-2\sum_{i=1}^{p}d_{i+2}u_i-\sum_{i=1}^{p}u_i^2. \]
Since $d_{i+2}\ge d_2$ and $u_i^2\ge u_i$,
\[ 2\deg(R/H_\epsilon)\le d_1^2+d_2^2-3e^2-(2d_2+1)u. \tag{8} \]
Since $a,b,c$ are linearly independent, there is no syzygy of coefficient degree zero, and hence $d_1\ge e+1$. On the other hand, the Koszul relation $(b,-a,0)$ has coefficient degree $e$, so $d_1\le2e$. Write
\[ d_1=e+x,\qquad d_2=e+y,\qquad 1\le x\le e,\qquad x\le y. \]
By (7), $u=x+y-e$, so $y=e+u-x$. Substituting into (8), we obtain
\[ 2\deg(R/H_\epsilon)\le2e^2-2x(e-x)-u(u+1). \tag{9} \]
If $1\le x\le e-1$, then $x(e-x)\ge e-1$ and $u(u+1)\ge2$, so $2\deg(R/H_\epsilon)\le2e^2-2e$. If $x=e$, then $u=y\ge e$, and (9) gives the same inequality. Therefore,
\[ \deg(R/H_\epsilon)\le e^2-e, \] as desired.
\end{proof}

\begin{cor}\label{cor:refined-bound}
Combining all the results of this section yields
\[ \operatorname{Bour}(F)\le
\begin{cases}
 e^2-e+1-\ell_{\epsilon},& H_\epsilon\text{ perfect},\\[2mm]
 e^2-e-\ell_{\epsilon},& H_\epsilon\text{ non-perfect}.
\end{cases}
\]
Consequently, in both cases,
\[ \operatorname{Bour}(F) \le e^2-e+1-\ell_{\epsilon}.
\]
\end{cor}

\begin{proof}
From the proof of Theorem \ref{prop:basic-bound},
\[ \operatorname{Bour}(F) \le\deg(R/H_\epsilon)-|V(H_\epsilon)\setminus V(I_\epsilon)|. \]
Now apply Theorem \ref{prop:perfect-bound} or Proposition \ref{prop:nonperfect-bound}, according to the case.
\end{proof}

We next show that the previous bounds are sharp. 

\begin{exa}
The first bound in Corollary \ref{cor:refined-bound} can be attained. Consider the quartic curve
\[ F=x^2y^2+x^2z^2+y^2z^2+x^2yz+2xy^2z+2xyz^2, \]
which also appears in \cite{JNS2024}*{Ex. 2.13}. A minimal syzygy of $J_F$ is
\[ \epsilon= \bigl( xy-y^2+yz+2z^2,\, 2xy+2yz+2z^2,\, -2y^2-xz-yz+z^2 \bigr),\]
and $\deg(\epsilon)=e=2$. Set $H_\epsilon=\langle a,b,c\rangle$, where $\epsilon=(a,b,c)$. Moreover, the forms $a,b,c$ are linearly independent over $k$. Using \textsc{Singular}, we obtain the minimal graded free resolution
\[ 0\longrightarrow R(-3)^2 \longrightarrow R(-2)^3 \longrightarrow H_\epsilon \longrightarrow0. \]
Hence $H_\epsilon$ is a perfect ideal of height $2$. The Bourbaki ideal associated to $\epsilon$ is $I_\epsilon=\langle -2y^2+2z^2,-x-y-z\rangle$. Computations in \textsc{Singular} also give
\[ V(H_\epsilon)=\{(0:1:-1),(-2:1:1),(1:0:0)\}, \qquad V(I_\epsilon)=\{(0:1:-1),(-2:1:1)\}. \]
Therefore $\ell_{\epsilon}=1$. Since $e=2$, Corollary \ref{cor:refined-bound} gives
\[ \operatorname{Bour}(F) \leq e^2-e+1-\ell_{\epsilon} = 4-2+1-1 = 2. \]
On the other hand, $I_\epsilon$ has minimal graded free resolution
\[ 0\longrightarrow R(-3) \longrightarrow R(-1)\oplus R(-2) \longrightarrow I_\epsilon \longrightarrow0. \]
It follows that $\deg(R/I_\epsilon)=2$, and hence $\operatorname{Bour}(F)=2$. Thus equality holds in the bound of Corollary \ref{cor:refined-bound}, when $H_\epsilon$ is perfect.
\end{exa}

\begin{exa}
The refined bound in Corollary \ref{cor:refined-bound} can also be attained in the non-perfect case. Consider the reduced curve
\[
F=x(x^3-y^3)(xy+z^2).
\]
Using \textsc{Singular}, we obtain the minimal syzygy $\epsilon=(-xz,-yz,3xy+2z^2)$ of degree $e=2$. Thus
$H_\epsilon=\langle xz,yz,3xy+2z^2\rangle$. Clearly, the three entries of $\epsilon$ are linearly independent. Moreover, \textsc{Singular} gives the minimal graded free resolution
\[
0\longrightarrow R(-5)\longrightarrow R(-3)\oplus R(-4)^2\longrightarrow R^3(-2)\longrightarrow H_\epsilon\longrightarrow0,
\]
and hence $H_\epsilon$ is not perfect.
In this case, $I_\epsilon=\langle y,z\rangle$. Furthermore,
\[ V(H_\epsilon)=\{(1:0:0),(0:1:0)\},\qquad V(I_\epsilon)=\{(1:0:0)\}. \]
Therefore $\ell_\epsilon=1$, and Corollary \ref{cor:refined-bound} gives
\[ \operatorname{Bour}(F)\leq e^2-e-|V(H_\epsilon)\setminus V(I_\epsilon)|=4-2-1=1.\]
On the other hand, since $I_\epsilon =\langle y,z\rangle$, we have $\operatorname{Bour}(F)=\deg(R/I_\epsilon)=1$. Thus the second refined bound in Corollary \ref{cor:refined-bound} is attained.
\end{exa}

We finish this section with another upper bound, which does not involve the ideal $H_{\epsilon}$ or the integer $\ell_{\epsilon}$. It depends only on the number of minimal generators of $\operatorname{Syz}(J_F)$ that occur in its initial degree.

\begin{prop}
Let $F$ be a reduced plane curve of degree $d+1$, and let $d_1\leq d_2\leq \cdots \leq d_m$ be the degrees of a minimal generating set of $\operatorname{Syz}(J_F)$. If $e = d_1 = d_2$, then 
\[ \operatorname{Bour}(F)\leq e(2e-d). \]
Moreover, if $m\geq 4$, then
\[ \operatorname{Bour}(F)\leq e(2e-d)-1. \]
\end{prop}

\begin{proof} If $F$ is a free curve, there is nothing to prove. In the following, we suppose $m \geq 3$. 
By \cite{Dimca4}*{Thm. 3.5}, there exists an integer $d'$ satisfying $d_3\leq d'\leq \min\{d_m,d\}$ such that
\[ \tau(F)\geq d(d-e)+e^2-(e+d_2-d)(e+d'-d). \]
Using $\operatorname{Bour}(F)=d^2+e(e-d)-\tau(F)$, we obtain
\[ \operatorname{Bour}(F) \leq (e+d_2-d)(e+d'-d). \]
Since $d_2=e$ and $d'\leq d$, it follows that
\[\operatorname{Bour}(F)\leq (2e-d)(e+d'-d)\leq e(2e-d).\]
If $m\geq4$, the inequality in \cite{Dimca4}*{Theorem 3.5} is strict, and therefore, 
\[ \operatorname{Bour}(F)\leq e(2e-d)-1, \]
which completes the proof.
\end{proof}

\begin{exa}
Consider the reduced plane curve $F=(x^4+y^4)(2x^4+y^4+2x^2z^2+z^4)$. A direct computation using \cite{Singular} gives minimal syzygy degrees $(d_1,d_2,d_3)=(4,4,7)$, and hence $e=4$. Using the minimal free resolution of $R/J_F$, one has $\tau(F) = \deg (R/J_F) = 33$. 

Therefore, $\operatorname{Bour}(F)=4=4(2\cdot4-7)=e(2e-d)$. Hence equality holds in the first bound of Proposition~5.11.
\end{exa}


\section{Realization of the Bourbaki Degree} \label{sec:realization}

In this section, we study the realization of Bourbaki degrees for fixed degree curves. Our aim is to answer the following question: given $d\in\mathbb{N}$ and $n\in\{0,1,\ldots,d^2\}$, does there exist a reduced curve $F$ of degree $d+1$ such that $\operatorname{Bour}(F)=n$? That is, for a fixed degree, all the possible Bourbaki degrees are attained. 

For the reader’s convenience, we recall the du Plessis–Wall bounds, which we will use repeatedly below. 

\begin{thm} 
    Let $F$ be a reduced plane projective curve of degree $d+1$ and $\indeg = e$. Then 
\[
d(d-e)\leq \tau(F)\leq
\begin{cases}
d(d-e)+e^2, & 2e\leq d\\[3pt]
d(d-e)+e^2-\displaystyle\binom{2e-d+1}{2}, & 2e>d
\end{cases}.
\]
\end{thm}

\begin{proof}
    See \cite{duPlessisWall}*{Thm. 3.2}.
\end{proof}

\subsection{Curves of low degree.} We begin by considering small values of $d$. Throughout, $e=\indeg$. If $d=0$, then $F$ is a line. Write $F=ax+by+cz$ and assume, without loss of generality, that $a\neq0$. Clearly, $(-b,a,0)$ and $(-c,0,a)$ are nonzero syzygies of $J_F$. Moreover, every $(u,v,w)\in\operatorname{Syz}(J_F)$ satisfies $(u,v,w)=\frac{v}{a}(-b,a,0)+\frac{w}{a}(-c,0,a)$. Therefore, $\operatorname{Syz}(J_F)=\langle(-b,a,0),(-c,0,a)\rangle$. Hence $F$ is free and $\operatorname{Bour}(F)=0$.

If $d=1$, the possible Bourbaki degrees are $0$ and $1$. Let $F$ be a smooth conic. By \cite{JNS2024}*{Prop. 2.14}, we have $\operatorname{Bour}(F)=1$. The singular case realizes Bourbaki degree $0$ (for example, the union of two lines).

For $d=2,3$, we use the du Plessis-Wall bounds to exclude certain values.
If $d=2$, then $F$ has degree $3$ and $0\leq e\leq2$. The du Plessis-Wall bound gives $\tau(F)\geq2(2-e)$, and the formula for the Bourbaki degree yields $\operatorname{Bour}(F)=4+e^2-2e-\tau(F)\leq e^2$. Thus, if $e\leq1$, we have $\operatorname{Bour}(F)\leq1$. On the other hand, if $e=2$, the refined du Plessis-Wall upper bound gives $\tau(F)\leq1$, so $\operatorname{Bour}(F)=4-\tau(F)\geq3$. Therefore, no reduced curve of degree $3$ has Bourbaki degree $2$. By \cite{Daniel}*{p. 33}, all remaining values of Bourbaki degree are realized. 

If $d=3$, then $F$ has degree $4$ and $0\leq e\leq3$. Again, the du Plessis-Wall bound implies $\operatorname{Bour}(F)\leq e^2$. Hence, if $e\leq2$, we have $\operatorname{Bour}(F)\leq4$. If $e=3$, the du Plessis-Wall bound gives $\tau(F)\leq3$, and therefore $\operatorname{Bour}(F)=9-\tau(F)\geq6$. Thus, no reduced curve of degree $4$ has Bourbaki degree $5$. By \cite{Daniel}*{p. 34-36}, all remaining values of Bourbaki degree are realized.

\begin{rem}
As seen above, the du Plessis--Wall bounds impose numerical obstructions to the Bourbaki degree when $d=2,3$. For $d\geq 4$, no such obstruction occurs. Indeed, using $\operatorname{Bour}(F)=d(d-e)+e^2-\tau(F)$, the numerically admissible integers for a fixed $1\leq e\leq d$ form $[0,e^2]\cap\mathbb{Z}$ if $e\leq \lfloor d/2\rfloor$, and $\bigl[\binom{2e-d+1}{2},e^2\bigr]\cap\mathbb{Z}$ otherwise. Since $\binom{2e-d+1}{2}\leq(e-1)^2+1$ for $\lfloor d/2\rfloor+1\leq e\leq d$, these intervals leave no integer gaps and their union is $[0,d^2]\cap\mathbb{Z}$.
\end{rem}

\subsection{Curves of higher degree.} We now turn to the case $d\geq4$. We show that every integer from $0$ to $d^2$ is realized as the Bourbaki degree of a reduced curve of degree $d+1$. To this end, we prove the following results.

\begin{lem}\label{prop:bour-zero}
	Let $d\geq4$ and $1\leq e\leq d/2$ be integers. Then there exists a reduced curve $F$ in $\mathbb{P}^2$ of degree $d+1$ such that $\indeg=e$ and $\operatorname{Bour}(F)=0$.
\end{lem}
\begin{proof}
	Set $D:=d+1$ and $m:=D-e$. Consider the line arrangement
	\[F=x\prod_{j=1}^{e}(x-jy)\prod_{i=1}^{m-1}(x-u_i z),\]
	where $u_1,\ldots,u_{m-1}$ are pairwise distinct nonzero elements in $k$. Then $\deg F=1+e+m-1=e+m=D=d+1$, and $F$ is reduced.

	The lines $x=0$ and $x-jy=0$, for $j=1,\ldots,e$, pass through $A=(0:0:1)$. Thus $A$ is an ordinary point of multiplicity $e+1$. The lines $x=0$ and $x-u_i z=0$, for $i=1,\ldots,m-1$, pass through $B=(0:1:0)$. Hence $B$ is an ordinary point of multiplicity $m$.

	All remaining intersections are ordinary nodes. There are exactly $e(m-1)$ such points, arising from the intersections of the lines $x-jy=0$ with the lines $x-u_i z=0$.

	Since $m\geq e+1$, we have $2m\geq D+1$. By \cite{Dimca-line-arrangements}*{Thm. 1.2}, it follows that $\indeg=D-m=e$. An ordinary point formed by $r$ distinct concurrent lines has Tjurina number $(r-1)^2$, and each node contributes $1$. Therefore,
	\[\tau(F)=e^2+(m-1)^2+e(m-1).\]
	On the other hand, since $d=e+m-1$, we have
	\[d(d-e)+e^2=(e+m-1)(m-1)+e^2=e^2+(m-1)^2+e(m-1).\]
	Consequently, $\operatorname{Bour}(F)=d^2+e(e-d)-\tau(F)=0$, as required.
\end{proof}

\begin{lem}\label{prop:bour-dois}
	Let $d\geq4$ be an integer. Then there exists a reduced curve $F$ in $\mathbb{P}^2$ of degree $d+1$ such that $\indeg=2$ and $\operatorname{Bour}(F)=2$.
\end{lem}
\begin{proof}
	Set $D:=d+1$, $e=2$, and $m:=D-e=d-1$. We construct a curve of the form $F=GH$, where $G$ is the product of $m$ distinct lines through a point $P$ and $H$ is a smooth conic with $P\notin V(H)$.

	After a linear change of coordinates, we may assume that $P=(0:0:1)$. Let $G=\prod_{j=1}^{m}(a_jx+b_jy)$ be a product of $m$ distinct lines through $P$, chosen generically so that they meet $H$ transversely. Set $F=GH$. Then $\deg F=m+2=D=d+1$. Moreover, $F$ is reduced, since $G$ and $H$ are reduced and have no common irreducible factor.

	The singularities of $F$ are the point $P$, formed by the $m$ distinct concurrent lines, and the $2m$ nodes arising from their transverse intersections with the conic $H$. The point $P$ has Tjurina number $(m-1)^2$, and each node contributes $1$. Thus,
	\[\tau(F)=(m-1)^2+2m=(d-2)^2+2(d-1)=d^2-2d+2.\]

	We claim that $\indeg=2$. Since $G_z=0$, we have $F_z=GH_z$. Consider $\varepsilon=(xH_z,yH_z,zH_z-DH)$, whose entries are homogeneous forms of degree $2$. Euler's formula yields
\begin{align*}
	xH_zF_x+yH_zF_y+(zH_z-DH)F_z&=H_zH(xG_x+yG_y)+GH_z(xH_x+yH_y+zH_z-DH)\\ &=(m+2-D)GHH_z=0.
\end{align*}
	This syzygy is nonzero: if $H_z\neq0$, its first entry is nonzero and if $H_z=0$, its third entry is $-DH\neq0$. Hence $r:=\indeg\leq2$.

	If $r\leq1$, the du Plessis--Wall lower bound  would give $d^2-2d+2=\tau(F)\geq d(d-r)\geq d(d-1)$, and therefore $d\leq2$, contradicting the fact that $d\geq4$. Thus $\indeg=2$. Finally,
	\[\operatorname{Bour}(F)=d^2+2(2-d)-\tau(F)=d^2+4-2d-(d^2-2d+2)=2.\]
	Therefore, the curve $F$  has the required properties.
\end{proof}

\begin{lem}\label{prop:bour-ultima-faixa}
	Let $d\geq4$ and $1\leq e\leq d-1$ be integers. For every integer $n$ with $e(e-1)+1\leq n\leq e^2$, there exists a reduced curve $F$ in $\mathbb{P}^2$ of degree $d+1$ such that $\indeg=e$ and $\operatorname{Bour}(F)=n$.
\end{lem}
\begin{proof}
	Write $n=e^2-\delta$, where $0\leq\delta\leq e-1$. It suffices to construct a reduced curve $F$ of degree $d+1$ with $\indeg=e$ and $\tau(F)=d(d-e)+\delta$, since then $\operatorname{Bour}(F)=e^2-\delta=n$.

	If $e=1$, then $\delta=0$. Take $F=x^dy+z^{d+1}$, which is reduced. Its partial derivatives are $F_x=dx^{d-1}y$, $F_y=x^d$, and $F_z=(d+1)z^d$. The tuple $(x,-dy,0)$ is a syzygy of degree $1$, while the linear independence of $F_x,F_y,F_z$ excludes syzygies of degree $0$. Hence $\indeg=1$. The equations $F_y=F_z=0$ show that $P=(0:1:0)$ is the only singular point. In the chart $y=1$, its local equation is $f_P(s,t)=s^d+t^{d+1}$, and its Tjurina ideal is $\langle s^{d-1},t^d\rangle$ in $k[s,t]_{\langle s,t\rangle}$. Thus $\tau(F)=d(d-1)$ and $\operatorname{Bour}(F)=1=n$.

	Suppose that $e\geq2$, and set $D:=d+1$ and $m:=D-e\geq2$. For a parameter $\lambda\in k$ with $\lambda\neq0$, and $\lambda\neq-1$ if $e=2$, define
	\[Q_{\delta,\lambda}=x^{e-2}y^2+\lambda y^e-\sum_{j=0}^{\delta}\binom{D}{j}x^{e-j}(z-x)^j,\qquad F=F_{\delta,\lambda}:=x^mQ_{\delta,\lambda}+z^D.\]
	Then $\deg F=D$. Moreover, $F_y=x^m(2x^{e-2}y+e\lambda y^{e-1})$ is a nonzero polynomial in $k[x,y]$, so any common irreducible factor of $F$ and $F_y$ would belong to $k[x,y]$. Since $\deg_z Q_{\delta,\lambda}\leq\delta<D$, the coefficient of $z^D$ in $F$ is $1$, and no nonconstant polynomial in $k[x,y]$ can divide $F$. Thus $F$ and $F_y$ have no common irreducible factor, proving that $F$ is reduced.

	At $Q=(1:0:1)$, use coordinates $u=y/x$ and $v=z/x-1$. By the binomial formula, $f_Q(u,v)=F(1,u,1+v)=A_\lambda(u)+B_\delta(v)$, where $A_\lambda(u)=u^2+\lambda u^e$ and $B_\delta(v)=\sum_{j=\delta+1}^{D}\binom{D}{j}v^j$. We have $A_\lambda'(u)=u(2+e\lambda u^{e-2})$, and the second factor is a unit at the origin: its value is $2$ if $e\geq3$ and $2(1+\lambda)\neq0$ if $e=2$. Also, $B_\delta(v)=v^{\delta+1}C_\delta(v)$ with $C_\delta(0)=\binom{D}{\delta+1}\neq0$. Thus $B_\delta'(v)=v^\delta((\delta+1)C_\delta(v)+vC_\delta'(v))$ is $v^\delta$ times a unit. It follows that the local derivatives generate $\langle u,v^\delta\rangle$. Since $f_Q$ also belongs to this ideal, its Tjurina ideal is $\langle u,v^\delta\rangle$ in $k[u,v]_{\langle u,v\rangle}$. For $\delta\geq1$, the quotient has basis $\{1,v,\ldots,v^{\delta-1}\}$, and for $\delta=0$ it is zero. Hence $\tau_Q(F)=\delta$, with $Q$ smooth when $\delta=0$.

	At $P=(0:1:0)$, the coordinates $s=x/y$ and $t=z/y$ give $f_P(s,t)=s^mU(s,t)+t^D$, where $U(s,t)=Q_{\delta,\lambda}(s,1,t)$. Since $\delta\leq e-1$, we have $U(0,0)=\lambda$ if $e\geq3$ and $U(0,0)=1+\lambda$ if $e=2$. Thus $U(0,0)\neq0$. Now $(f_P)_s=s^{m-1}(mU+sU_s)$, with $mU+sU_s$ a unit, and $(f_P)_t=s^mU_t+Dt^{D-1}$. The first derivative therefore generates $\langle s^{m-1}\rangle$ in $A=k[s,t]_{\langle s,t\rangle}$. Since $s^mU_t\in\langle s^{m-1}\rangle$, the two derivatives generate $\langle s^{m-1},t^{D-1}\rangle A$. As $f_P$ also belongs to this ideal,
	\[\tau_P(F)=\dim_{k}\frac{A}{\langle s^{m-1},t^{D-1}\rangle A}=(m-1)(D-1)=d(d-e).\]
	We next choose $\lambda$ so that there are no other singularities. Since $F(0,y,z)=z^D$, the only point of $F$ on $x=0$ is $P$. In the chart $x\neq0$, singular points satisfy $A_\lambda(u)+B_\delta(v)=A_\lambda'(u)=B_\delta'(v)=0$. If $v=0$ and $\delta=0$, there is no solution, since $B_0'(0)=D$. If $\delta\geq1$, the conditions reduce to $A_\lambda(u)=A_\lambda'(u)=0$, which forces $u=0$. Indeed, for $e=2$ this follows from $\lambda\neq-1$. For $e\geq3$, a common nonzero root would give both $\lambda u^{e-2}=-1$ and $\lambda u^{e-2}=-2/e$, a contradiction. Thus $Q$ is the only possible singular point with $v=0$.

	For $v\neq0$, write $B_\delta(v)=(1+v)^D-P_\delta(v)$, where $P_\delta(v)=\sum_{j=0}^{\delta}\binom{D}{j}v^j$. Using $(j+1)\binom{D}{j+1}=(D-j)\binom{D}{j}$ gives
	\[DB_\delta(v)-(1+v)B_\delta'(v)=(1+v)P_\delta'(v)-DP_\delta(v)=-(D-\delta)\binom{D}{\delta}v^\delta.\]
	Since $\delta<D$, the coefficient on the right is nonzero. Hence $B_\delta'(v)=0$ implies $B_\delta(v)\neq0$, and a singular point must have $u\neq0$. If $e=2$, this contradicts $A_\lambda'(u)=2(1+\lambda)u=0$. If $e\geq3$, the equation $A_\lambda'(u)=0$ gives $\lambda u^{e-2}=-2/e$, so $A_\lambda(u)=\frac{e-2}{e}u^2$. Consequently, $u^2=-\frac{e}{e-2}B_\delta(v)$ and $\lambda=-\frac{2}{eu^{e-2}}$. The polynomial $B_\delta'$ is independent of $\lambda$ and has finitely many nonzero roots. For each such root, the first equality gives at most two nonzero values of $u$, and the second determines $\lambda$. Hence, only finitely many values of $\lambda$ can produce additional singularities. Choosing $\lambda$ outside this finite set, while retaining the initial restrictions, gives $\tau(F)=d(d-e)+\delta$.

	Finally, write $Q=Q_{\delta,\lambda}$. Since $F_x=x^{m-1}(mQ+xQ_x)$ and $F_y=x^mQ_y$, then $\varepsilon=(xQ_y,-mQ-xQ_x,0)$ satisfies $(xQ_y)F_x+(-mQ-xQ_x)F_y=0$. It is nonzero, since $Q_y\neq0$, and has degree $e$. Hence $r:=\indeg\leq e$. If $r\leq e-1$, the du Plessis-Wall bound gives $\tau(F)\geq d(d-r)\geq d(d-e)+d>d(d-e)+\delta$, contradicting the value of $\tau(F)$ computed above. Thus, $\indeg=e$, and $\operatorname{Bour}(F)=d^2+e(e-d)-\tau(F)=e^2-\delta=n$.
\end{proof}

\begin{lem}\label{prop:preenchimento-entre-quadrados}
	Let $d\geq4$ and $3\leq e\leq d-1$ be integers. For every integer $n$ with $(e-1)^2+1\leq n\leq e(e-1)$, there exists a reduced projective plane curve $F$ of degree $d+1$ such that $\indeg=e$ and $\operatorname{Bour}(F)=n$.
\end{lem}
\begin{proof}
	Set $\delta=e(e-1)-n$. Then $0\leq\delta\leq e-2$. Since $e\geq3$, we have $e-2\leq\binom{e-1}{2}$. By \cite{severi}, there exists an irreducible nodal curve $H$ of degree $e$ with exactly $\delta$ nodes. If $\delta=0$, take $H$ to be smooth.

	Choose $P\in\mathbb{P}^2\setminus V(H)$ and set $m:=d+1-e\geq2$. After a linear change of coordinates, we may assume that $P=(0:0:1)$. Choose $m$ distinct lines $L_1,\ldots,L_m$ through $P$ so that none passes through a node of $H$ and all meet $H$ transversely. Set $Q:=\prod_{i=1}^{m}L_i$ and $F:=QH$. Since all the lines pass through $P$, we have $Q\in k[x,y]$, and therefore $Q_z=0$.

	Clearly, $\deg F=m+e=d+1$. Moreover, $F$ is reduced, since $H$ is irreducible, the lines $L_i$ are distinct, and none is a component of $H$.

	The singularities of $F$ are the singularities of the nodal curve $H$, the point $P$, and the $me$ nodes arising from the transversal intersections of the lines $L_i$ with $H$. Each of the $\delta$ nodes of $H$ contributes $1$ to the Tjurina number, the point $P$ contributes $(m-1)^2$, and each intersection node contributes $1$. Therefore, $\tau(F)=(m-1)^2+\delta+me$.

	Consider $\varepsilon=(xH_z,yH_z,zH_z-(m+e)H)$, whose entries are homogeneous forms of degree $e$. A direct computation using Euler's formula shows that $\varepsilon\in\operatorname{Syz}(J_F)$. Moreover, $\varepsilon$ is nonzero. If $H_z\neq0$, its first entry is nonzero. If $H_z=0$, its third entry is $-(m+e)H\neq0$. Hence $\indeg\leq e$. We claim that equality holds.

	Set $\widetilde{H}:=HL_1L_2$. This curve has degree $e+2$, has exactly three irreducible components, and is nodal. Indeed, besides the nodes of $H$ and its transverse intersections with $L_1$ and $L_2$, the intersection of these two lines at $P$ is also a node, since $P\notin V(H)$. By \cite{Dimca-Sernesi}*{Ex. 2.2(i)}, we have $\operatorname{Syz}(J_{\widetilde{H}})_j=0$ for all $j\leq e-1$ and $\dim_{k}\operatorname{Syz}\left(J_{\widetilde{H}}\right)_e=3-1=2$. Consequently, $\operatorname{indeg}(\operatorname{Syz}(J_{\widetilde{H}}))=e$.

	If $m=2$, then $F=\widetilde{H}$, and the desired equality has already been proved. If $m\geq3$, write $F=\widetilde{H}\,\widetilde{Q}$, where $\widetilde{Q}:=L_3\cdots L_m$. Since $\widetilde{H}$ and $\widetilde{Q}$ have no common irreducible factor, \cite{Burity-Simis-Tohaneanu}*{Prop. 3.2} gives $\indeg\geq\operatorname{indeg}\left(\operatorname{Syz}\left(J_{\widetilde{H}}\right)\right)=e$. Combining the two inequalities, we conclude that $\indeg=e$.

	Finally, by the formula for the Bourbaki degree and the equality $m=d+1-e$ 
	\begin{align*} \operatorname{Bour}(F)&=d^2+e^2-ed-\bigl((m-1)^2+\delta+me\bigr)\\ &=d^2+e^2-ed-(d-e)^2-\delta-(d+1-e)e\\ &=e(e-1)-\delta=n. \end{align*}
	Therefore, the curve $F$ has all the required properties.
\end{proof}

\begin{lem}\label{prop:e-d}
	Let $d\geq4$ be an integer. For every integer $n$ with $\binom{d+1}{2}\leq n\leq d^2$, there exists a reduced plane curve $F$ of degree $d+1$ such that $\indeg=d$ and $\operatorname{Bour}(F)=n$.
\end{lem}
\begin{proof}
	Set $u=d^2-n$. Since $0\leq u\leq\binom{d}{2}$, by \cite{severi} there exists an irreducible nodal curve $F$ of degree $d+1$ with exactly $u$ nodes. Each node contributes $1$ to the Tjurina number, so $\tau(F)=u$. By \cite{Dimca-Sernesi}*{Ex. 2.2(i)}, the irreducibility of $F$ implies that there are no nonzero syzygies of degree at most $d-1$. On the other hand, the Koszul relations yield a nonzero syzygy of degree $d$. Therefore, $\indeg=d$. The formula for the Bourbaki degree now gives $\operatorname{Bour}(F)=d^2-u=n$. Since $d^2-\binom{d}{2}=\binom{d+1}{2}$, this is precisely the lower endpoint of the stated interval.
\end{proof}

We finish this article proving Theorem \ref{thm-C} of the introduction. 

\begin{thm}\label{teo:realizacao-completa-bourbaki}
	Let $d\geq4$ be an integer. Then, for every integer $n$ with $0\leq n\leq d^2$, there exists a reduced projective plane curve $F$ of degree $d+1$ such that $\operatorname{Bour}(F)=n$.
\end{thm}
\begin{proof}
	Lemma \ref{prop:bour-zero}, with $e=1$, yields $n=0$. Lemma \ref{prop:bour-dois} yields $n=2$.
Lemma \ref{prop:bour-ultima-faixa}, applied with $e=1$, realizes $n=1$. Lemma \ref{prop:bour-ultima-faixa} applied with $e=2$ realizes $n=3$ and $n=4$.

	Now fix $e$ with $3\leq e\leq d-1$. Lemma \ref{prop:preenchimento-entre-quadrados} realizes every integer in $[(e-1)^2+1,e(e-1)]$, while Lemma \ref{prop:bour-ultima-faixa} realizes every integer in $[e(e-1)+1,e^2]$. These two intervals are consecutive, so every integer from $(e-1)^2+1$ to $e^2$ is realized.

	As $e$ ranges from $3$ to $d-1$, the integers in the intervals $[(e-1)^2+1,e^2]$ are consecutive, starting at $5$ and ending at $(d-1)^2$. Thus, we have already realized all the values between $0$ and $(d-1)^2$.

	On the other hand, Lemma \ref{prop:e-d} realizes every integer in $[\binom{d+1}{2},d^2]$. Since
	\[(d-1)^2+1-\binom{d+1}{2}=\frac{(d-1)(d-4)}{2}\geq0,\]
	there is no gap between this interval and the values already realized. Therefore, every integer $n$ with $0\leq n\leq d^2$, is realized as the Bourbaki degree of a plane projective curve of fixed degree $d+1$.
\end{proof}

    
	\section*{Acknowledgments}

    The first author thanks Marcos Jardim for introducing the problem and for many fruitful discussions. The second author thanks Aldício Miranda for valuable assistance with the use of the Singular software and Michelle Morgado for fruitful discussions. The authors also thank Abbas Nejad for useful suggestions on the first version of this manuscript. 
		
		\bibliographystyle{plain}
		\bibliography{biblio}

\end{document}